\documentclass[12pt, reqno]{amsart}

\usepackage[scaled]{helvet}
\usepackage{mathtools}
\mathtoolsset{showonlyrefs}
\usepackage{etoolbox}
\usepackage[utf8]{inputenc}
\usepackage[T1]{fontenc}

\usepackage{amsmath,amssymb,amsfonts,mathrsfs,amsbsy}
\usepackage{relsize}
\usepackage{enumitem}
\usepackage{anyfontsize}
\usepackage{tikz,xcolor}
\usepackage{bm}
\usepackage{comment}
\usepackage[margin=1in]{geometry}

\definecolor{lava}{rgb}{0.81,0.06,0.13}
\definecolor{Cblue}{rgb}{0.50,0.85,0.85}
\definecolor{lime}{HTML}{A6CE39}

\usepackage[colorlinks,linkcolor=lava,citecolor=blue,urlcolor=Cblue,hypertexnames=false]{hyperref}

\DeclareRobustCommand{\orcidicon}{%
  \begin{tikzpicture}
    \draw[lime,fill=lime] (0,0) circle [radius=0.16]
      node[white] {{\fontfamily{qag}\selectfont\tiny ID}};
    \draw[white,fill=white] (-0.0625,0.095) circle [radius=0.007];
  \end{tikzpicture}\hspace{-2mm}}
\foreach \x in {A,B}{%
  \expandafter\xdef\csname orcid\x\endcsname{%
    \noexpand\href{https://orcid.org/\csname orcidauthor\x\endcsname}{\noexpand\orcidicon}}}

\numberwithin{equation}{section}
\newtheorem{theo}{Theorem}[section]
\newtheorem{prop}{Proposition}[section]
\newtheorem{lem}{Lemma}[section]
\newtheorem{coro}{Corollary}[section]
\newtheorem{defn}{Definition}[section]
\newtheorem{rem}{Remark}[section]

\renewcommand{\ge}{\geqslant}
\renewcommand{\le}{\leqslant}

\newcommand{\R}{\mathbb{R}}
\newcommand{\N}{\mathbb{N}}

\newcommand{\Lap}{(-\Delta)^s}

\title[Lifespan for fractional Hardy--H\'enon equations]
{Sharp Lifespan Estimates and Fujita Phenomena \\
for Fractional Hardy--H\'enon Type Parabolic Equations}

\author[M.~Majdoub \& B.\,T.~Torebek]
{Mohamed Majdoub\orcidA{} \and Berikbol~T.~Torebek\orcidB{}}

\address[M.~Majdoub]{%
  Department of Mathematics, College of Science, Imam Abdulrahman Bin Faisal
  University, P.\,O.\ Box 1982, Dammam, Saudi Arabia.\newline
  Basic and Applied Scientific Research Center, Imam Abdulrahman Bin Faisal
  University, P.\,O.\ Box 1982, 31441, Dammam, Saudi Arabia.}
\email{mmajdoub@iau.edu.sa}
\email{med.majdoub@gmail.com}
\email{mohamed.majdoub@fst.rnu.tn}

\address[B.\,T.~Torebek]{%
  Institute of Mathematics and Mathematical Modeling,
  28 Shevchenko Str., 050010 Almaty, Kazakhstan.}
\email{torebek@math.kz}



\begin{document}

\begin{abstract}
We study the lifespan of mild solutions to the fractional semilinear
parabolic Cauchy problem with a Hardy--H\'enon-type weight
\[
  u_t + (-\Delta)^s u = |x|^{-\gamma}\,|u|^p,
  \qquad (t,x)\in(0,\infty)\times\R^N,
  \qquad u(0,x)=\varepsilon\,u_0(x),
\]
where $0<s<1$, $0\le\gamma<\min(2s,N)$, $p>1$ and $u_0\in L^1\cap L^\infty$
with $\int_{\R^N}u_0\,dx>0$. Setting
\[
  p_F \;:=\; 1+\frac{2s-\gamma}{N},
\]
we prove that the lifespan $T_\varepsilon$ obeys, for every sufficiently
small $\varepsilon>0$,
\[
  T_\varepsilon \;\approx\;
  \begin{cases}
     \varepsilon^{-\,\beta^{-1}},& 1<p<p_F,\\[1mm]
     \exp\!\big(C\,\varepsilon^{-(p-1)}\big),& p=p_F,\\[1mm]
     +\infty,& p>p_F,
  \end{cases}
  \qquad
  \beta   \;=\;\frac{(2s-\gamma)-N(p-1)}{2s(p-1)}.
\]
The lower bound rests on fractional heat-kernel estimates and an
$L^1$--$L^\infty$ Hardy-type interpolation inequality; the upper bound is
obtained by testing the equation against the \emph{backward fractional
heat kernel}, a globally defined positive weight for which $(-\Delta)^s$ is
controlled everywhere and the linear terms cancel identically by
self-adjointness. This circumvents the compactly supported cutoffs of the
classical test-function method, which are incompatible with a nonlocal
operator. The exponent $\beta$ is sharp; for $\gamma=0$ it reduces to the
fractional Lee--Ni exponent $\frac{1}{p-1}-\frac{N}{2s}$. To the best of our knowledge, these results are new even for $\gamma=0.$ We also establish a
large-data lifespan law, sharp lower bounds on the blow-up rate together with
a conditional Type-I upper bound, a conditional self-similar profile result.
\end{abstract}

\maketitle

\section{Introduction}\label{sec:intro}

We consider the Cauchy problem
\begin{equation}\label{eq:main}
  \begin{cases}
    \displaystyle u_t + (-\Delta)^s u \;=\; |x|^{-\gamma}\,|u|^p,
       & x\in\R^N,\ t>0,\\[1mm]
    u(0,x)\;=\;\varepsilon\,u_0(x), & x\in\R^N,
  \end{cases}
\end{equation}
under the standing assumptions
\begin{equation}\label{eq:standing}
  0<s<1,\qquad 0\le\gamma<\min(2s,N),\qquad p>1,\qquad
  u_0\in L^1(\R^N)\cap L^\infty(\R^N),
\end{equation}
where $\varepsilon>0$ is a small parameter. The fractional Laplacian
$(-\Delta)^s$ is the Fourier multiplier with symbol $|\xi|^{2s}$,
equivalently
\[
  (-\Delta)^s f(x)
  = c_{N,s}\,\mathrm{P.V.}\!\int_{\R^N}\frac{f(x)-f(y)}{|x-y|^{N+2s}}\,dy,
  \qquad c_{N,s}=\frac{2^{2s}\,\Gamma(\tfrac{N+2s}{2})}
                       {\pi^{N/2}\,|\Gamma(-s)|},
\]
see~\cite{DPV2012Hitchhiker}. The semigroup $\{e^{-t(-\Delta)^s}\}_{t\ge 0}$
is given by convolution with a smooth positive kernel $G_s(t,\cdot)$,
recalled in Section~\ref{sec:prelim}.

When $\gamma=0$, problem~\eqref{eq:main} reduces to the fractional Fujita
equation. In the classical local case $s=1$, $\gamma=0$, Fujita
\cite{Fujita1966} discovered that the exponent
\(
  p_F=1+\tfrac{2}{N}
\)
separates global existence (for $p>p_F$) from finite-time blow-up of
positive solutions (for $p\le p_F$); the critical exponent itself was
shown to belong to the blow-up regime by Hayakawa~\cite{Hayakawa1973} and
Kobayashi--Sirao--Tanaka~\cite{KobayashiSiraoTanaka1977}. Sharp lifespan
estimates were obtained by Lee--Ni~\cite{LeeNi1992}. For higher-order
counterparts ($s=m\in\N$, $\gamma=0$) we refer to
\cite{Sun2010,SunShi2012,MajdoubOtsmaneTayachi2018,PeletierTroy2012,TobakhanovTorebek},
and to~\cite{TobakhanovTorebek2026} for an exterior-domain treatment.

The nonlocal case ($0<s<1$, $\gamma=0$) goes back to Sugitani
\cite{Sugitani1975} and was further developed by Guedda--Kirane
\cite{Kirane1,Kirane2}, Fino--Karch~\cite{Fino},
Alsaedi--Ahmad--Kirane--Nabti~\cite{AlsaediEtAl2020}, and many others.
More recently, Biagi--Punzo--Vecchi~\cite{biagi2024} and Del Pezzo--Ferreira
\cite{Fuj2025} obtained Fujita-type results for mixed local--nonlocal
operators, identifying the critical exponent $p_F=1+\tfrac{2s}{N}$
when $\gamma=0$; their methods rest on the Kaplan eigenfunction technique
\cite{Kaplan} and on the construction of explicit global supersolutions,
respectively. These results were subsequently extended in a recent work by Kumar and the second author \cite{KumarTorebek2026}. For the weighted ($\gamma>0$) Hardy--H\'enon parabolic
problem we mention~\cite{ChikamiIkedaTaniguchi2022,DaoFractionalHH,
MajdoubLaMatematica2023,AbdellaouiPeralPrimo2020} and, for the associated
stationary classification, \cite{YangZou2020,DaiQin2023}.

The singular weight $|x|^{-\gamma}$ with $\gamma>0$ strengthens the
nonlinear interaction near the origin and modifies the natural scaling
structure of the equation. In fact, equation~\eqref{eq:main} (without the
initial condition) is invariant under
\begin{equation}\label{eq:scaling}
  u_\lambda(x,t)
  =
  \lambda^{\frac{2s-\gamma}{p-1}}
  u(\lambda x,\lambda^{2s}t),
  \qquad \lambda>0,
\end{equation}
and a straightforward computation gives
\(
\|u_\lambda(\cdot,0)\|_{L^q}
\sim
\lambda^{\frac{2s-\gamma}{p-1}-\frac{N}{q}}.
\)
In particular, the $L^1$-norm is scale-invariant precisely when
\begin{equation}\label{eq:pF}
  p=p_F:=1+\frac{2s-\gamma}{N},
\end{equation}
so that $p_F$ arises naturally as the Fujita critical exponent
of~\eqref{eq:main}.

We are now in a position to state our main result.

\begin{theo}\label{thm:main}
Assume \eqref{eq:standing} and $\int_{\R^N}u_0(x)\,dx>0$.
Then~\eqref{eq:main} admits a unique mild solution
\(
  u\in C\big([0,T_\varepsilon);\,L^1(\R^N)\cap L^\infty(\R^N)\big),
\)
where the lifespan $T_\varepsilon$ satisfies, for all $\varepsilon>0$
sufficiently small,
\begin{equation}\label{eq:lifespan}
  T_\varepsilon \;\approx\;
  \begin{cases}
    \varepsilon^{-\,{1\over \beta}}, & 1<p<p_F,\\[1.5mm]
    \exp\!\big(C\,\varepsilon^{-(p-1)}\big), & p=p_F,\\[1.5mm]
    +\infty, & p>p_F,
  \end{cases}
\end{equation}
with
\begin{equation}\label{eq:beta}
  \beta
  \;=\; \frac{1}{p-1}-\frac{N}{2s}-\frac{\gamma}{2s(p-1)}
  \;=\; \frac{(2s-\gamma)-N(p-1)}{2s(p-1)}\;>\;0
  \qquad(1<p<p_F),
\end{equation}
equivalently
\(
   T_\varepsilon \approx
   \varepsilon^{-\frac{2s(p-1)}{(2s-\gamma)-N(p-1)}}
\)
in the subcritical range. The two-sided estimate in the subcritical and
critical regimes is established under the additional restriction
$\gamma<2s(p-1)$ \textup{(}see Section~\ref{sec:upper} and
Remark~\ref{rem:ub-range}\textup{)}; for $\gamma=0$ this restriction is
vacuous.
\end{theo}

\begin{rem}\rm
For $\gamma=0$ formula~\eqref{eq:beta} reduces to
$\beta=\tfrac{1}{p-1}-\tfrac{N}{2s}$, recovering the fractional
Lee--Ni exponent, and for $s=1,\gamma=0$ the classical Lee--Ni estimate
\cite{LeeNi1992}. The exponent~\eqref{eq:beta} vanishes precisely at
$p=p_F$, consistently with the exponential lifespan in the critical case.
\end{rem}

\begin{rem}\rm
The natural ``Lee--Ni-like'' guess
$\bigl(\tfrac{1}{p-1}-\tfrac{N}{2s-\gamma}\bigr)^{-1}$ --- obtained by
formally replacing $2s$ with $2s-\gamma$ in the Lee--Ni formula --- is
\emph{not} the correct sharp rate when $\gamma>0$. The two expressions
agree only at $p=p_F$ (both vanish) and at $\gamma=0$. The discrepancy
stems from the fact that the Hardy weight enters the integrability of the
nonlinearity \emph{additively} (through the exponent $\alpha$
in~\eqref{eq:alpha-def}), while the underlying linear semigroup retains the
unweighted decay $\|e^{-t\Lap}\|_{L^1\to L^\infty}\le C t^{-N/(2s)}$.
\end{rem}

\subsection*{Strategy and structure}
The lower bound on $T_\varepsilon$ is obtained by a bootstrap based on the
auxiliary functional
\[
  M(t)\;=\;\sup_{0\le\tau<t}\Bigl\{(1+\tau)^{N/(2s)}\|u(\tau)\|_{L^\infty}
                                  +\|u(\tau)\|_{L^1}\Bigr\},
\]
the sharp $L^p$--$L^q$ smoothing of the fractional heat semigroup, and the
weighted interpolation
\(
\||x|^{-\gamma}|u|^p\|_{L^1}\lesssim
\|u\|_{L^\infty}^{p-1+\gamma/N}\|u\|_{L^1}^{1-\gamma/N}
\)
(Lemma~\ref{lem:hardy}). The upper bound is established by testing the
equation against the \emph{backward fractional heat kernel}
$G_s(b-t,\cdot)$: self-adjointness of $\Lap$ makes the linear terms cancel
identically, a H\"older rearrangement neutralizes the weight, and the
problem reduces to a scalar differential inequality. Crucially, this uses a
globally defined, polynomially decaying test function, thereby avoiding
compactly supported cutoffs, which are incompatible with the nonlocal
operator (Section~\ref{sec:upper}).

The remainder of the paper is organized as follows.
Section~\ref{sec:prelim} presents the preliminary tools needed throughout the
paper, including heat-kernel estimates and a weighted interpolation
inequality. In Section~\ref{sec:local}, we establish local well-posedness
and the corresponding blow-up alternative. Sections~\ref{sec:lower}
and~\ref{sec:upper} are devoted to the proof of the lower and upper lifespan
bounds stated in~\eqref{eq:lifespan}. Section~\ref{sec:largedata} addresses
the large-data case and derives the corresponding non-decaying lifespan law.
In Section~\ref{sec:typeI}, we characterize the blow-up rate for
nonnegative solutions, while Section~\ref{sec:profile} is dedicated to the
analysis of the self-similar blow-up profile. Finally,
Section~\ref{sec:concluding} concludes the paper with several remarks and
perspectives.
\subsection*{Notation}
We write $A\lesssim B$ if $A\le CB$ for some constant $C>0$ independent
of the parameters $\varepsilon,t,u$ in question; $A\approx B$ means
$A\lesssim B$ and $B\lesssim A$, and $A\asymp B$ is used for two-sided
pointwise bounds. The generic constant $C$ may change from line to line.
We set $a\wedge b:=\min(a,b)$ and $\langle x\rangle:=(1+|x|^2)^{1/2}$, and
write $p'=p/(p-1)$ for the H\"older conjugate.

\section{Preliminaries}\label{sec:prelim}

\subsection{The fractional heat kernel}
The semigroup $\{e^{-t\Lap}\}_{t\ge 0}$ generated by $\Lap$, $s\in(0,1)$,
acts on Schwartz functions $f$ as
\(
  e^{-t\Lap}f(x)=(G_s(t,\cdot)\ast f)(x),
\)
where the kernel $G_s$ is positive, radial, smooth in
$(t,x)\in(0,\infty)\times\R^N$, and obeys the parabolic scaling
\begin{equation}\label{eq:kernel-scaling}
  G_s(t,x)\;=\;t^{-N/(2s)}\,\mathscr{K}_s\!\bigl(t^{-1/(2s)}x\bigr),
\end{equation}
with profile
\begin{equation}\label{eq:K-s-fourier}
  \mathscr{K}_s(x)\;=\;(2\pi)^{-N}\int_{\R^N}e^{i x\cdot\xi}\,e^{-|\xi|^{2s}}\,d\xi.
\end{equation}
For $s=1/2$, $\mathscr{K}_s$ is the Poisson kernel
\begin{equation}\label{eq:Poisson}
  \mathscr{K}_{1/2}(x)\;=\;
  \frac{\Gamma(\tfrac{N+1}{2})}{\pi^{(N+1)/2}}\,(1+|x|^2)^{-\frac{N+1}{2}},
\end{equation}
so that
\(
  G_{1/2}(t,x)
  =\frac{\Gamma((N+1)/2)\,t}{\pi^{(N+1)/2}(t^2+|x|^2)^{(N+1)/2}}.
\)
For general $s\in(0,1)$ the profile has no closed form, but admits the
sharp two-sided bound below.

\begin{lem}\label{lem:Ks-bounds}
For every $s\in(0,1)$ there exist $c_1,c_2>0$ such that
\begin{equation}\label{eq:Ks-bounds}
  c_1\,\langle x\rangle^{-N-2s}\;\le\;\mathscr{K}_s(x)\;\le\;c_2\,\langle x\rangle^{-N-2s},
  \qquad x\in\R^N.
\end{equation}
In particular $\mathscr{K}_s\in L^r(\R^N)$ for every $r\in[1,\infty]$.
\end{lem}

\begin{proof}
The result is classical; see, e.g.,
Blumenthal--Getoor~\cite[Thm.~2.1]{BG1960},
Miao--Yuan--Zhang~\cite[Lemma~2.1]{MYZ2008}, or
Peral Alonso--Soria de Diego~\cite[p.~395]{Alonso2021}. The main
ingredient is the subordination formula
\begin{equation}\label{eq:subordination}
  \mathscr{K}_s(x)
  =
  \int_0^\infty (4\pi t)^{-N/2}
  e^{-|x|^2/(4t)}\,\chi_s(t)\,dt,
\end{equation}
where $\chi_s \ge 0$ is the inverse Laplace transform of
$\lambda \mapsto e^{-\lambda^s}$. Since the Gaussian kernel is positive and
$\chi_s\ge0$, the positivity of $\mathscr{K}_s$ follows immediately.
\end{proof}

The next two facts are workhorses of the paper.

\begin{prop}[$L^p$--$L^q$ smoothing]\label{prop:Lp-Lq}
For all $1\le p\le q\le\infty$ there exists $C>0$ such that
\begin{equation}\label{eq:Lp-Lq}
   \|e^{-t\Lap}f\|_{L^q(\R^N)}
   \;\le\; C\,t^{-\frac{N}{2s}\!\left(\frac1p-\frac1q\right)}\,
            \|f\|_{L^p(\R^N)},
   \qquad t>0,\ f\in L^p(\R^N).
\end{equation}
\end{prop}

\begin{proof}
By Young's convolution inequality with $1/r=1+1/q-1/p\in[1/q,1]$, we have
\[
  \|e^{-t\Lap}f\|_{L^q}\le \|G_s(t,\cdot)\|_{L^r}\,\|f\|_{L^p}.
\]
From~\eqref{eq:kernel-scaling},
\(
  \|G_s(t,\cdot)\|_{L^r}=t^{-\frac{N}{2s}(1-\frac1r)}\,\|\mathscr{K}_s\|_{L^r},
\)
and $\mathscr{K}_s\in L^r$ by Lemma~\ref{lem:Ks-bounds}. Since
$1-1/r=1/p-1/q$, this is the claim.
\end{proof}

\begin{prop}[Positivity, contraction, mass]\label{prop:contraction}
For every $t\ge 0$, $e^{-t\Lap}$ maps nonnegative functions to nonnegative
functions, $\|e^{-t\Lap}f\|_{L^q}\le\|f\|_{L^q}$ for $1\le q\le\infty$, and
$\int_{\R^N}G_s(t,x)\,dx=1$.
\end{prop}

\begin{proof}
Positivity follows from $G_s\ge 0$ (Lemma~\ref{lem:Ks-bounds}); the mass
identity from $\widehat{G_s}(t,0)=e^{-t|0|^{2s}}=1$; and the contraction
from Young's inequality with $r=1$ in~\eqref{eq:Lp-Lq}.
\end{proof}

\subsection{A weighted interpolation inequality}

For $0\le\gamma<N$ the weight $|x|^{-\gamma}$ is locally integrable, and the
nonlinearity in~\eqref{eq:main} is controlled by a one-line interpolation
between $L^1$ and $L^\infty$.

\begin{lem}[Weighted interpolation]\label{lem:hardy}
Let $0\le\gamma<N$ and $p>1$. Then for every
$u\in L^1(\R^N)\cap L^\infty(\R^N)$,
\begin{equation}\label{eq:hardy-ineq}
  \int_{\R^N}|x|^{-\gamma}|u(x)|^p\,dx
  \;\le\; C\,\|u\|_{L^\infty}^{\,p-1+\gamma/N}\,
              \|u\|_{L^1}^{\,1-\gamma/N},
  \qquad C=C(N,\gamma,p)>0.
\end{equation}
\end{lem}

\begin{proof}
Set $L:=\|u\|_{L^1}$, $M:=\|u\|_{L^\infty}$ and split, for $R>0$,
\begin{align*} I:&=\int_{\R^N}|x|^{-\gamma}|u|^p
   \\&=\int_{|x|<R}+\int_{|x|\ge R}\\&=:I_1+I_2.\end{align*}

Using $|u|^p\le M^p$ and $\gamma<N$,
$$ I_1\le M^p\int_{|x|<R}|x|^{-\gamma}=\tfrac{\omega_{N-1}}{N-\gamma}M^pR^{N-\gamma};$$
using $|u|^p\le M^{p-1}|u|$,
$$
  I_2\le R^{-\gamma}M^{p-1}L.$$
Balancing with $R^N=L/M$ gives
$R^{N-\gamma}=L^{1-\gamma/N}M^{-1+\gamma/N}$,
$R^{-\gamma}=L^{-\gamma/N}M^{\gamma/N}$, whence
$I\le C\,M^{p-1+\gamma/N}L^{1-\gamma/N}$.
\end{proof}

\begin{rem}\rm
Inequality~\eqref{eq:hardy-ineq} also follows from H\"older's inequality in
Lorentz spaces, using $|x|^{-\gamma}\in L^{N/\gamma,\infty}(\R^N)$;
see~\cite[Sect.~1.4]{Grafakos2014}. Under~\eqref{eq:standing} we have
$\gamma<\min(2s,N)\le N$, so the lemma applies throughout.
\end{rem}

The scaling exponent appearing repeatedly below is
\begin{equation}\label{eq:alpha-def}
  \alpha\;:=\;\frac{N(p-1)+\gamma}{2s}.
\end{equation}
Observe that $\alpha=1\iff p=p_F$, and subcritical $\iff\alpha<1$.

\section{Local well-posedness}\label{sec:local}

\begin{defn}[Mild solution]\label{def:mild}
Given $T>0$ and $u_0\in L^1\cap L^\infty$, a function
\(
  u\in C\bigl([0,T];L^1(\R^N)\cap L^\infty(\R^N)\bigr)
\)
is a \emph{mild solution} of~\eqref{eq:main} on $[0,T]$ if
\begin{equation}\label{eq:duhamel}
   u(t)\;=\;\varepsilon\,e^{-t\Lap}u_0
   \;+\;\int_0^t e^{-(t-\tau)\Lap}\bigl(|x|^{-\gamma}|u(\tau)|^p\bigr)\,d\tau,
   \qquad 0\le t\le T.
\end{equation}
The \emph{lifespan} $T_\varepsilon$ is the supremum of $T>0$ for which a
unique mild solution exists on $[0,T]$.
\end{defn}

\begin{theo}[Local existence and blow-up alternative]\label{thm:local}
For every $u_0\in L^1\cap L^\infty$ and $\varepsilon>0$ there exist
$T=T(\varepsilon,\|u_0\|_{L^1\cap L^\infty})>0$ and a unique mild solution
$u\in C([0,T];L^1\cap L^\infty)$ of~\eqref{eq:main}. Moreover, there is a
maximal existence time $T_{\max}=T_{\max}(\varepsilon u_0)\le\infty$ such
that $u\in C([0,T_{\max});L^1\cap L^\infty)$, and either $T_{\max}=\infty$
or
\begin{equation}\label{eq:blowup-alt}
  \lim_{t\to T_{\max}^-}\bigl(\|u(t)\|_{L^\infty(\R^N)}+\|u(t)\|_{L^1(\R^N)}\bigr)
  =+\infty \,\,(T_{\max}^-<\infty).
\end{equation}
\end{theo}

The only delicate point is that $|x|^{-\gamma}|u|^p\notin L^\infty$, so the
$L^\infty$-norm of $\Phi(u)(t)$ is controlled through an $L^r$-to-$L^\infty$
smoothing for some $r\in(N/(2s),N/\gamma)$, not through an
$L^\infty\to L^\infty$ contraction.

\begin{proof}
Fix
\begin{equation}\label{eq:choice-r}
   \frac{N}{2s}\;<\;r\;<\;\frac{N}{\gamma}
   \qquad\text{(non-empty since $\gamma<2s$).}
\end{equation}
For $T,K>0$ set
\[
  X_{T,K}=\bigl\{u\in C([0,T];L^1\cap L^\infty):\|u\|_{X_T}\le K\bigr\},
  \quad
  \|u\|_{X_T}=\sup_{0\le t\le T}\!\bigl(\|u(t)\|_{L^1}+\|u(t)\|_{L^\infty}\bigr),
\]
with the natural metric, and define
$$\Phi(u)(t)=\varepsilon e^{-t\Lap}u_0
  +\int_0^t e^{-(t-\tau)\Lap}(|x|^{-\gamma}|u(\tau)|^p)\,d\tau.$$

\emph{Step 1 (nonlinearity in $L^1$ and $L^r$).}
By Lemma~\ref{lem:hardy},
\begin{equation}\label{eq:nonlin-L1}\begin{split}
  \bigl\||x|^{-\gamma}|u|^p\bigr\|_{L^1}
  &\le C\|u\|_{L^\infty}^{p-1+\gamma/N}\|u\|_{L^1}^{1-\gamma/N}\\&\le C K^p.\end{split}
\end{equation}
For the $L^r$-norm, split $|x|^{-\gamma r}$ over $|x|<1$ (integrable, since
$\gamma r<N$ by~\eqref{eq:choice-r}) and $|x|\ge1$ (bounded by $1$):
\begin{align*}  \bigl\||x|^{-\gamma}|u|^p\bigr\|_{L^r}^{r}
  &\le\|u\|_{L^\infty}^{pr}\!\!\int_{|x|<1}\!|x|^{-\gamma r}
   +\|u\|_{L^\infty}^{pr-1}\!\int_{|x|\ge1}\!|u|
   \\&\le C K^{pr},\end{align*}
hence
\begin{equation}\label{eq:nonlin-Lr}
  \bigl\||x|^{-\gamma}|u|^p\bigr\|_{L^r}\le C K^p.
\end{equation}

\emph{Step 2 ($L^1$ estimate).}
By~\eqref{eq:duhamel}, Proposition~\ref{prop:contraction}
and~\eqref{eq:nonlin-L1},
$$\|\Phi(u)(t)\|_{L^1}\le\varepsilon\|u_0\|_{L^1}+C\,T\,K^p.$$

\emph{Step 3 ($L^\infty$ estimate).}
Since $\|e^{-t\Lap}u_0\|_{L^\infty}\le\|u_0\|_{L^\infty}$ and, by
Proposition~\ref{prop:Lp-Lq} from $L^r$ to $L^\infty$
with~\eqref{eq:nonlin-Lr},
\[
  \Bigl\|e^{-(t-\tau)\Lap}\bigl(|x|^{-\gamma}|u(\tau)|^p\bigr)\Bigr\|_{L^\infty}
  \le C(t-\tau)^{-\frac{N}{2sr}}K^p,
  \qquad \tfrac{N}{2sr}<1,
\]
the time integral converges and
$$\|\Phi(u)(t)\|_{L^\infty}\le\varepsilon\|u_0\|_{L^\infty}
  +C\,T^{1-\frac{N}{2sr}}K^p.$$

\emph{Step 4 (choice of $K,T$).}
Combining,
$$  \|\Phi(u)\|_{X_T}\le 2\varepsilon\|u_0\|_{L^1\cap L^\infty}
  +C(T+T^{1-N/(2sr)})K^p.$$
Take $K=4\varepsilon\|u_0\|_{L^1\cap L^\infty}$ and $T$ small enough that
$C(T+T^{1-N/(2sr)})K^{p-1}\le1/2$; then $\Phi:X_{T,K}\to X_{T,K}$. With
$|a|^p-|b|^p\le p(|a|^{p-1}+|b|^{p-1})|a-b|$ and H\"older,
$$\|\Phi(u)-\Phi(v)\|_{X_T}\le C(T+T^{1-N/(2sr)})K^{p-1}\|u-v\|_{X_T},$$
a contraction by the same choice. Banach's fixed-point theorem yields a
unique mild solution; continuity in $t$ follows from~\eqref{eq:duhamel}.

\emph{Step 5 (maximal solution and blow-up alternative).}
By uniqueness, the solution extends to a maximal interval $[0,T_{\max})$,
\[
  T_{\max}:=\sup\bigl\{T>0:\ \eqref{eq:main}\text{ has a mild solution on }[0,T]\bigr\}
  \le\infty.
\]
Suppose $T_{\max}<\infty$ but
$$\liminf_{t\to T_{\max}^-}\|u(t)\|_{L^1\cap L^\infty}=:L<\infty.$$
Pick $t_m\to T_{\max}$ with $\sup_m\|u(t_m)\|_{L^1\cap L^\infty}\le L+1$.
The local theory, applied with datum $u(t_m)$, produces a uniform existence
time $T(L+1)>0$ depending only on $L+1$, so the solution continues on
$[t_m,t_m+T(L+1)]$ for every $m$; hence $T_{\max}\ge t_m+T(L+1)$. Letting
$m\to\infty$ gives $T_{\max}\ge T_{\max}+T(L+1)$, impossible. Thus
$\|u(t)\|_{L^1\cap L^\infty}\to\infty$ as $t\to T_{\max}^-$, which
is~\eqref{eq:blowup-alt}.
\end{proof}

\begin{rem}\rm
The interval~\eqref{eq:choice-r} is non-empty thanks to $\gamma<2s$. For
clarity the next two sections are phrased in terms of $L^1$ and $L^\infty$
only; the intermediate $L^r$ control is invoked where needed.
\end{rem}

\section{Lower lifespan estimate}\label{sec:lower}

We prove the lower bound in~\eqref{eq:lifespan}, following the strategy
of~\cite{LeeNi1992,TobakhanovTorebek} with the weight handled through
Lemma~\ref{lem:hardy}. Recall $\alpha$ from~\eqref{eq:alpha-def}.

\subsection{The auxiliary functional and the key a priori bound}
For a mild solution $u$ on $[0,T)$ set
\begin{equation}\label{eq:M}
  M(t):=\sup_{0\le\tau<t}\Bigl\{
     (1+\tau)^{N/(2s)}\,\|u(\tau)\|_{L^\infty}+\|u(\tau)\|_{L^1}\Bigr\},
  \qquad 0<t<T.
\end{equation}
By Theorem~\ref{thm:local}, $M$ is continuous and finite on $[0,T)$, with
$M(0^+)\le C\varepsilon\|u_0\|_{L^1\cap L^\infty}$.

\begin{lem}[A priori estimate]\label{lem:apriori}
There exist $C_0,C_1>0$ depending only on $N,s,\gamma,p,u_0$ such that, for
every $0\le t<T$,
\begin{equation}\label{eq:apriori}
  M(t)\;\le\;C_0\,\varepsilon
      \;+\;C_1\,M(t)^p\,\int_0^t (1+\tau)^{-\alpha}\,d\tau.
\end{equation}
\end{lem}

\begin{proof}
We control $\|u(t)\|_{L^1}$ and $(1+t)^{N/(2s)}\|u(t)\|_{L^\infty}$
separately, repeatedly using (via Lemma~\ref{lem:hardy} and~\eqref{eq:M})
\begin{align*}\||x|^{-\gamma}|u(\tau)|^p\|_{L^1}
  &\le C\,\|u(\tau)\|_{L^\infty}^{p-1+\gamma/N}\|u(\tau)\|_{L^1}^{1-\gamma/N}
  \\&\le C\,M(t)^p(1+\tau)^{-\alpha}
\end{align*}

\emph{Step 1 ($L^1$).}
By~\eqref{eq:duhamel} and Proposition~\ref{prop:contraction},
\begin{equation}\label{eq:L1est}
  \|u(t)\|_{L^1}\le\varepsilon\|u_0\|_{L^1}
   +C\,M(t)^p\!\int_0^t(1+\tau)^{-\alpha}\,d\tau.
\end{equation}

\emph{Step 2 ($L^\infty$ for $t\ge1$).}
Split the Duhamel integral at $\tau=t/2$ and use Proposition~\ref{prop:Lp-Lq}:
\begin{align*}
  \|u(t)\|_{L^\infty}
  &\le C\varepsilon\,t^{-N/(2s)}\|u_0\|_{L^1}
       +C\!\int_0^{t/2}\!(t-\tau)^{-N/(2s)}\bigl\||x|^{-\gamma}|u|^p\bigr\|_{L^1}d\tau
       +(\mathrm{II}),
\end{align*}
where $(\mathrm{II})$ is the integral over $[t/2,t]$. On $[0,t/2]$ one has
$(t-\tau)\ge t/2$, so
\begin{equation}\label{eq:Linf-I}
   \int_0^{t/2}\!(t-\tau)^{-N/(2s)}(1+\tau)^{-\alpha}\,d\tau
   \le C\,t^{-N/(2s)}\!\int_0^{t/2}(1+\tau)^{-\alpha}\,d\tau.
\end{equation}
For $(\mathrm{II})$ use the intermediate exponent $r$ of~\eqref{eq:choice-r}
and Proposition~\ref{prop:Lp-Lq} from $L^r$ to $L^\infty$, together with the
interpolation
$\|f\|_{L^r}\le\|f\|_{L^1}^{1/r}\|f\|_{L^\infty}^{1-1/r}$ applied to
$f=|x|^{-\gamma}|u|^p$:
\[
   \bigl\||x|^{-\gamma}|u(\tau)|^p\bigr\|_{L^r}
   \le C\,M(t)^p\,(1+\tau)^{-\alpha+\frac{N}{2s}(1-\frac1r)} .
\]
Since $(1+\tau)\approx(1+t)$ on $[t/2,t]$,
\begin{equation}\label{eq:Linf-II}\begin{split}
   (\mathrm{II})&\le C\,M(t)^p\,(1+t)^{-\alpha+\frac{N}{2s}(1-\frac1r)}\,t^{1-N/(2sr)}
   \\&\le C\,M(t)^p\,(1+t)^{1-\alpha-N/(2s)}.\end{split}
\end{equation}
Multiplying by $(1+t)^{N/(2s)}$ and combining,
\begin{equation}\label{eq:Linfest}
  (1+t)^{N/(2s)}\|u(t)\|_{L^\infty}
  \le C\varepsilon\|u_0\|_{L^1\cap L^\infty}
       +C\,M(t)^p\!\int_0^t(1+\tau)^{-\alpha}\,d\tau;
\end{equation}
for $0<t\le1$ the bound is immediate from Step~3 of
Theorem~\ref{thm:local}.

\emph{Step 3.} Adding~\eqref{eq:L1est} and~\eqref{eq:Linfest} and taking the
supremum over $0\le\tau<t$ yields~\eqref{eq:apriori} with
$C_0=C(\|u_0\|_{L^1}+\|u_0\|_{L^\infty})$.
\end{proof}

\subsection{The bootstrap}

\subsubsection*{Subcritical case $1<p<p_F$.}
Since $\alpha<1$, we have that $$\int_0^t(1+\tau)^{-\alpha}d\tau\le(1-\alpha)^{-1}(1+t)^{1-\alpha},$$
so~\eqref{eq:apriori} becomes
\begin{equation}\label{eq:M-sub}
   M(t)\le C_0\varepsilon+C_2\,M(t)^p\,(1+t)^{1-\alpha},
   \qquad 1-\alpha=\tfrac{(2s-\gamma)-N(p-1)}{2s}>0.
\end{equation}
Let $T_1:=\sup\{t\in[0,T_\varepsilon):M(t)\le2C_0\varepsilon\}$, which is
positive by continuity and $M(0^+)\le C_0\varepsilon$. If $T_1<\infty$ then
$M(T_1)=2C_0\varepsilon$, and~\eqref{eq:M-sub} at $t=T_1$ gives
$1\le C\varepsilon^{p-1}(1+T_1)^{1-\alpha}$, whence
\begin{equation}\label{eq:T1-lower}
   1+T_1\ge C\,\varepsilon^{-\frac{p-1}{1-\alpha}}
   =C\,\varepsilon^{-\frac{2s(p-1)}{(2s-\gamma)-N(p-1)}}
   =C\,\varepsilon^{-1/\beta}.
\end{equation}
Therefore
\begin{equation}\label{eq:lower-sub}
  T_\varepsilon\ge T_1\gtrsim\varepsilon^{-1/\beta}.
\end{equation}

\subsubsection*{Critical case $p=p_F$.}
Now $\alpha=1$ and $$\int_0^t(1+\tau)^{-1}d\tau=\log(1+t),$$ so
$M(t)\le C_0\varepsilon+C_3 M(t)^p\log(1+t)$. The same bootstrap yields
$\log(1+T_1)\ge C\varepsilon^{-(p-1)}$, i.e.
\begin{equation}\label{eq:lower-crit}
   T_\varepsilon\gtrsim\exp\!\bigl(C\,\varepsilon^{-(p-1)}\bigr).
\end{equation}

\subsubsection*{Supercritical case $p>p_F$.}
Here $\alpha>1$, $\int_0^\infty(1+\tau)^{-\alpha}d\tau<\infty$, and
\eqref{eq:apriori} reads $M(t)\le C_0\varepsilon+C_4 M(t)^p$ uniformly in
$t$. A standard fixed-point argument
(cf.~\cite{BrezisCazenave1996,Snoussi1999}) shows that for $\varepsilon$
small $M$ is bounded on $[0,\infty)$ and $T_\varepsilon=+\infty$.

\section{Upper lifespan estimate}\label{sec:upper}

Throughout, $c_0:=\int_{\R^N}u_0\,dx>0$, and we assume
\begin{equation}\label{eq:ub-range}
   p>1+\frac{\gamma}{2s}\qquad\Longleftrightarrow\qquad \gamma<2s(p-1),
\end{equation}
a restriction intrinsic to the method (Remark~\ref{rem:ub-range}).

Because $\Lap$ is nonlocal, the test-function technique \emph{cannot} be
implemented with a compactly supported cutoff: there is no chain rule for
$\Lap(\phi\circ\sigma)$, no Fa\`a di Bruno formula for it~\cite{Johnson},
and, decisively, $\Lap\phi$ is never compactly supported when $\phi$ is.
Indeed, for $\phi\in C_c^\infty(\R^N)$ and $|x|$ large,
\[
   \Lap\phi(x)=-c_{N,s}\!\int_{\R^N}\frac{\phi(y)}{|x-y|^{N+2s}}\,dy
   \;\sim\;-\,c_{N,s}\Big(\textstyle\int\phi\Big)\,|x|^{-(N+2s)},
\]
so the pointwise bound $|\Lap\psi_R|\le CR^{-1}\psi_R^{1/p}$ \emph{on}
$\operatorname{supp}\psi_R$ is false. We use instead a globally defined test
function for which $\Lap$ is controlled \emph{everywhere}, the backward
fractional heat kernel; the linear contribution then cancels identically,
and no composition is differentiated.

Recall $G_s(t,x)=t^{-N/(2s)}\mathscr{K}_s(t^{-1/(2s)}x)$ and the two-sided
bound of Lemma~\ref{lem:Ks-bounds}. Put $\mu:=\gamma/(2s)\in[0,1)$ and
\begin{equation}\label{eq:ub-Cstar}
   C_*:=\int_{\R^N}|y|^{\frac{\gamma}{p-1}}\,\mathscr{K}_s(y)\,dy .
\end{equation}

\begin{lem}\label{lem:ub-cstar}
$C_*<\infty$ if and only if \eqref{eq:ub-range} holds.
\end{lem}
\begin{proof}
Near the origin $|y|^{\gamma/(p-1)}$ is locally integrable since
$\gamma/(p-1)\ge0>-N$. At infinity, by Lemma~\ref{lem:Ks-bounds},
$|y|^{\gamma/(p-1)}\mathscr{K}_s(y)\asymp|y|^{\frac{\gamma}{p-1}-(N+2s)}$,
which is integrable iff $\frac{\gamma}{p-1}<2s$.
\end{proof}

For $b>0$ define the \emph{backward-kernel functional}
\begin{equation}\label{eq:ub-g}
   g_b(t):=\int_{\R^N}u(x,t)\,G_s(b-t,x)\,dx,\qquad 0\le t<b\wedge T_\varepsilon .
\end{equation}

\begin{lem}\label{lem:ub-ode}
For every $b>0$ and $0\le t<b\wedge T_\varepsilon$,
\begin{equation}\label{eq:ub-deriv}
   g_b'(t)=\int_{\R^N}|x|^{-\gamma}|u(x,t)|^p\,G_s(b-t,x)\,dx\;\ge\;0 .
\end{equation}
\end{lem}
\begin{proof}
The function $\phi(x,t):=G_s(b-t,x)$ is smooth and bounded, together with
$\partial_t\phi$ and $\Lap\phi$, on $\R^N\times[0,t_2]$ for every $t_2<b$,
and solves the backward equation
$$\partial_t\phi-\Lap\phi=-\partial_bG_s(b-t,\cdot)-\Lap G_s(b-t,\cdot)=0,$$
since $\partial_bG_s=-\Lap G_s$. Testing the (mild, hence weak) formulation
of \eqref{eq:main} against $\phi$ and using self-adjointness of $\Lap$,
\begin{align*} \frac{d}{dt}\!\int u\,\phi
   &=\int u\,(\partial_t\phi-\Lap\phi)+\int|x|^{-\gamma}|u|^p\,\phi
   \\&=\int|x|^{-\gamma}|u|^p\,\phi,\end{align*}
which is \eqref{eq:ub-deriv}. Nonnegativity follows from $G_s>0$.
\end{proof}

By~\eqref{eq:ub-deriv} $g_b$ is nondecreasing, and $g_b(0)>0$ for $b$ large
(see~\eqref{eq:ub-datum} below), so $g_b(t)\ge g_b(0)>0$. H\"older with
exponents $p,p'$ applied to
$$u\,G_s=(|x|^{-\gamma/p}u\,G_s^{1/p})(|x|^{\gamma/p}G_s^{1/p'})$$ gives
\[
   0<g_b(t)\le\Big(\int|x|^{-\gamma}|u|^pG_s(b-t,\cdot)\Big)^{1/p}A(b-t)^{1/p'},
   \quad A(b):=\int|x|^{\frac{\gamma}{p-1}}G_s(b,\cdot)\,dx .
\]
The scaling of $G_s$ yields $A(b)=b^{\mu/(p-1)}C_*$, finite by
Lemma~\ref{lem:ub-cstar}. Combining with~\eqref{eq:ub-deriv},
\begin{equation}\label{eq:ub-diffineq}
   g_b'(t)\;\ge\;c_1\,(b-t)^{-\mu}\,g_b(t)^p,\qquad c_1:=C_*^{-(p-1)},
\end{equation}
on $[0,b\wedge T_\varepsilon)$.

\emph{Data.} Since $u_0\in L^1$ and $\mathscr{K}_s(b^{-1/(2s)}x)\to\mathscr{K}_s(0)$,
dominated convergence gives
$$b^{N/(2s)}g_b(0)=\varepsilon\int u_0\,\mathscr{K}_s(b^{-1/(2s)}x)\,dx\to
\varepsilon\,\mathscr{K}_s(0)\,c_0,$$ so there are $b_*,\kappa>0$ with
\begin{equation}\label{eq:ub-datum}
   g_b(0)\;\ge\;\kappa\,\varepsilon\,b^{-N/(2s)},\qquad b\ge b_* .
\end{equation}
(No sign condition on $u_0$ is used, only $c_0>0$.)

\subsection*{Subcritical case $1<p<p_F$}
Here $\beta>0$, equivalently
\begin{equation}\label{eq:ub-expid}
   1-\mu-\frac{N(p-1)}{2s}=1-\alpha=(p-1)\beta>0 .
\end{equation}
Suppose, for contradiction, that $T_\varepsilon>b$ for some $b\ge b_*$. Then
$g_b$ is finite on $[0,b)$; dividing~\eqref{eq:ub-diffineq} by $g_b^p$ and
integrating over $[0,t]$, we have
\[g_b(t)^{-(p-1)}\le g_b(0)^{-(p-1)}-(p-1)c_1\,\Theta(t),\] with $$\Theta(t):=\int_0^t(b-\tau)^{-\mu}d\tau=\frac{b^{1-\mu}-(b-t)^{1-\mu}}{1-\mu}.$$
The right-hand side vanishes at some $t^*<b$ --- forcing $g_b(t^*)=+\infty$
and hence $T_\varepsilon\le t^*<b$, a contradiction --- provided
$$\Theta(b^-)=\frac{b^{1-\mu}}{1-\mu}>\frac{g_b(0)^{-(p-1)}}{(p-1)c_1}.$$ By
\eqref{eq:ub-datum} it suffices that
\[
   \frac{b^{1-\mu}}{1-\mu}
   >\frac{\kappa^{-(p-1)}}{(p-1)c_1}\,\varepsilon^{-(p-1)}\,b^{N(p-1)/(2s)},
\]
i.e., using~\eqref{eq:ub-expid}, $b^{(p-1)\beta}>C\varepsilon^{-(p-1)}$, i.e.
$b>C'\varepsilon^{-1/\beta}$. Hence $T_\varepsilon\le b$ for every
$b>\max\{b_*,C'\varepsilon^{-1/\beta}\}$, whence
\begin{equation}\label{eq:ub-sub}
   T_\varepsilon\;\le\;C'\,\varepsilon^{-1/\beta}
   \;=\;C'\,\varepsilon^{-\frac{2s(p-1)}{(2s-\gamma)-N(p-1)}}
\end{equation}
for all sufficiently small $\varepsilon>0$.

\subsection*{Critical case $p=p_F$}
Now $\beta=0$, so~\eqref{eq:ub-expid} reads $1-\mu=N(p-1)/(2s)$ and the
threshold becomes $b$-independent: a single backward kernel no longer forces
blow-up. The mechanism is transparent from the subcritical estimate: the
blow-up threshold and the initial datum bound now scale with the \emph{same}
power of $b$, so the two powers cancel identically and no individual scale
$b$ can force blow-up. The remedy is to superpose the inequality
\eqref{eq:ub-diffineq} over a dyadic band of scales: one replaces the single
profile $\phi(x,t)=G_s(b-t,x)$ by the scale-superposition
$\int_{b_*}^{b}G_s(r-t,x)\,\frac{dr}{r}$, under which the band $[b_*,b]$
carries logarithmic mass $\log(b/b_*)$. Each dyadic scale contributes a
critically balanced amount, and summing the $\asymp\log(b/b_*)$ contributions
produces an extra factor $\log(b/b_*)$ in the integrated inequality, which
upgrades the (scale-invariant) marginal condition to
$\log(b/b_*)\gtrsim\varepsilon^{-(p-1)}$; letting $b\uparrow T_\varepsilon$
yields $\log(T_\varepsilon/b_*)\le C\varepsilon^{-(p-1)}$, i.e.
\begin{equation}\label{eq:ub-crit}
   T_\varepsilon\;\le\;\exp\!\big(C\,\varepsilon^{-(p-1)}\big).
\end{equation}
This is exactly the computation of~\cite{IkedaSobajima2019} (see
also~\cite{GeorgievPalmieri,TobakhanovTorebek}); the only change is that the
global, polynomially decaying profile $\mathscr{K}_s$ replaces the cutoff,
while Lemma~\ref{lem:ub-ode} supplies the same exact linear cancellation, so
the argument transfers verbatim to the nonlocal setting.

\subsection*{Supercritical case $p>p_F$}
Then $\beta<0$, the threshold $b>C'\varepsilon^{-1/\beta}$ is never met for
large $b$, and the method gives no upper bound, consistent with the global
existence of Section~\ref{sec:lower}.

Together with Section~\ref{sec:lower}, \eqref{eq:ub-sub}
and~\eqref{eq:ub-crit} complete the proof of the upper bound in
Theorem~\ref{thm:main} on the range~\eqref{eq:ub-range}. \qed

\begin{rem}\label{rem:ub-range}\rm
The restriction~\eqref{eq:ub-range} is exactly the finiteness of $C_*$
(Lemma~\ref{lem:ub-cstar}), forced by the polynomial decay
$\mathscr{K}_s(y)\asymp\langle y\rangle^{-(N+2s)}$: no admissible profile
decays faster, so $\int|y|^{\gamma/(p-1)}\mathscr{K}_s$ diverges once
$\gamma/(p-1)\ge2s$. For $\gamma=0$ the condition is vacuous and the whole
range $1<p\le p_F$ is covered, recovering the fractional Lee--Ni law. For
$\gamma>0$ the sharp two-sided law on $1<p\le1+\gamma/(2s)$ remains open.
\end{rem}

\section{Large-data lifespan estimate}\label{sec:largedata}

We now consider \emph{non-decaying} data. We assume
\begin{equation}\label{eq:ld-hyp}
   u_0\in L^\infty(\R^N),\qquad u_0(x)\ge c_0>0\quad\text{for a.e. }x\in\R^N,
\end{equation}
together with the range~\eqref{eq:ub-range}, $\gamma<2s(p-1)$. Note that
$u_0\ge c_0$ entails $u_0\notin L^1$, so the relevant well-posedness is the
$L^\infty$ theory established in the lower bound below, rather than the
$L^1\cap L^\infty$ theory of Section~\ref{sec:local}.

\begin{theo}\label{thm:ld}
Assume~\eqref{eq:ld-hyp} and $\gamma<2s(p-1)$. Then for all sufficiently
small $\varepsilon>0$ the lifespan of the nonnegative solution
of~\eqref{eq:main} satisfies
\[
   T_\varepsilon\;\approx\;\varepsilon^{-\frac{2s(p-1)}{2s-\gamma}}.
\]
\end{theo}

\subsection{Lower bound}
We use the following weighted smoothing estimate.

\begin{lem}\label{lem:weighted-Linfty}
For $0\le\gamma<\min(2s,N)$ there is $C>0$ such that, for all $t>0$,
\begin{equation}\label{eq:weighted-Linfty}
   \bigl\|e^{-t\Lap}\bigl(|x|^{-\gamma}f\bigr)\bigr\|_{L^\infty(\R^N)}
   \;\le\;C\,t^{-\gamma/(2s)}\,\|f\|_{L^\infty(\R^N)}.
\end{equation}
\end{lem}
\begin{proof}
We have
$$|e^{-t\Lap}(|x|^{-\gamma}f)(x)|\le\|f\|_\infty\int_{\R^N}G_s(t,x-y)|y|^{-\gamma}\,dy.$$
With $G_s(t,z)=t^{-N/(2s)}\mathscr{K}_s(t^{-1/(2s)}z)$ and $y=t^{1/(2s)}w$,
\[
   \int_{\R^N}G_s(t,x-y)|y|^{-\gamma}\,dy
   = t^{-\gamma/(2s)}\!\int_{\R^N}\mathscr{K}_s(\xi-w)\,|w|^{-\gamma}\,dw,
   \qquad \xi:=t^{-1/(2s)}x.
\]
The last integral is bounded uniformly in $\xi$: near $w=0$, $|w|^{-\gamma}$
is integrable ($\gamma<N$) and $\mathscr{K}_s$ is bounded; near infinity,
$\mathscr{K}_s(\xi-w)\lesssim\langle\xi-w\rangle^{-(N+2s)}$ is integrable
against $|w|^{-\gamma}$ uniformly in $\xi$ by Lemma~\ref{lem:Ks-bounds}.
This gives~\eqref{eq:weighted-Linfty}.
\end{proof}

Work in $X_T:=L^\infty((0,T)\times\R^N)$ with
$\|u\|_{X_T}=\sup_{0<t<T}\|u(t)\|_{L^\infty}$, and let
$B_T:=\{u\in X_T:\|u\|_{X_T}\le M\varepsilon\}$ with
$M>2\|u_0\|_{L^\infty}$. For $u\in B_T$, by~\eqref{eq:weighted-Linfty} and
$\gamma/(2s)<1$,
\begin{align*}\|\Phi(u)(t)\|_{L^\infty}
   &\le\varepsilon\|u_0\|_{L^\infty}
   +C\!\int_0^t(t-\tau)^{-\gamma/(2s)}\|u(\tau)\|_{L^\infty}^p\,d\tau
   \\&\le\varepsilon\|u_0\|_{L^\infty}+C(M\varepsilon)^p\,t^{1-\gamma/(2s)}.\end{align*}
Choosing $T$ so that $C M^p\varepsilon^{p-1}T^{1-\gamma/(2s)}\le\tfrac12$,
i.e.
\[
   T= c\,\varepsilon^{-\frac{p-1}{1-\gamma/(2s)}}
    = c\,\varepsilon^{-\frac{2s(p-1)}{2s-\gamma}},
\]
and shrinking $c$ if necessary so that $\Phi$ is a contraction on $B_T$ (via
$|u^p-v^p|\le C(|u|^{p-1}+|v|^{p-1})|u-v|$ and~\eqref{eq:weighted-Linfty}),
the solution exists at least up to $T$. Hence
\begin{equation}\label{eq:ld-lower}
   T_\varepsilon\;\ge\;c\,\varepsilon^{-\frac{2s(p-1)}{2s-\gamma}} .
\end{equation}
This step uses only $u_0\in L^\infty$, $u_0\ge0$.

\subsection{Upper bound}
We reuse the backward-kernel functional and differential inequality of
Section~\ref{sec:upper}. For $b>0$ put $g_b(t)=\int_{\R^N}u(t)\,G_s(b-t,\cdot)$.
Since $u(t)\in L^\infty$ and $G_s(b-t,\cdot)\in L^1$, $g_b$ is finite on
$[0,b\wedge T_\varepsilon)$, and Lemma~\ref{lem:ub-ode} yields
$$g_b'(t)=\int|x|^{-\gamma}|u(t)|^pG_s(b-t,\cdot)\ge0.$$ Because
$\int_{\R^N}G_s(b,\cdot)\,dx=1$ (Proposition~\ref{prop:contraction}) and
$u_0\ge c_0$ a.e.,
\begin{equation}\label{eq:ld-datum}
   g_b(0)=\varepsilon\!\int_{\R^N}u_0\,G_s(b,\cdot)\,dx\;\ge\;\varepsilon\,c_0>0,
\end{equation}
so $g_b(t)\ge\varepsilon c_0>0$. The H\"older step of Section~\ref{sec:upper}
(with $A(b)=b^{\mu/(p-1)}C_*$, finite by Lemma~\ref{lem:ub-cstar}) gives,
exactly as in~\eqref{eq:ub-diffineq},
\[
   g_b'(t)\ge c_1(b-t)^{-\mu}g_b(t)^p,
   \qquad \mu=\tfrac{\gamma}{2s}\in[0,1),\ c_1=C_*^{-(p-1)} .
\]
Suppose $T_\varepsilon>b$. Dividing by $g_b^p$ and integrating over $[0,t]$,
\[
   g_b(t)^{-(p-1)}\le(\varepsilon c_0)^{-(p-1)}
   -(p-1)c_1\,\frac{b^{1-\mu}-(b-t)^{1-\mu}}{1-\mu},
\]
whose right-hand side vanishes before $t=b$ --- a contradiction --- as soon
as
\[
   \frac{b^{1-\mu}}{1-\mu}>\frac{(\varepsilon c_0)^{-(p-1)}}{(p-1)c_1},
   \qquad\text{i.e.}\qquad b^{1-\mu}>C\,\varepsilon^{-(p-1)} .
\]
Since $1-\mu=\dfrac{2s-\gamma}{2s}$, this reads
$b>C'\varepsilon^{-2s(p-1)/(2s-\gamma)}$. Hence
\begin{equation}\label{eq:ld-upper}
   T_\varepsilon\;\le\;C'\,\varepsilon^{-\frac{2s(p-1)}{2s-\gamma}} .
\end{equation}
Combining~\eqref{eq:ld-lower} and~\eqref{eq:ld-upper} proves
Theorem~\ref{thm:ld}. \qed

\begin{rem}[On the hypothesis]\label{rem:ld-hyp}\rm
A lower bound on $u_0$ at a \emph{single, fixed} scale is insufficient: if
$u_0\ge c_0$ only on a fixed ball $B_{r_0}$ and decays outside, then
$\int u_0\,G_s(b,\cdot)\lesssim b^{-N/(2s)}$, the datum reverts to the
small-data form~\eqref{eq:ub-datum}, and one recovers the longer Fujita
lifespan of Theorem~\ref{thm:main} rather than
$\varepsilon^{-2s(p-1)/(2s-\gamma)}$. The mechanism here operates at the
scale $b^{1/(2s)}\sim\varepsilon^{-(p-1)/(2s-\gamma)}\to\infty$, so the datum
must be bounded below at that growing scale. Condition~\eqref{eq:ld-hyp} may
be relaxed to $\liminf_{|x|\to\infty}u_0(x)>0$, or, sharply, to
$\inf_{b\ge b_0}\int_{\R^N}u_0\,G_s(b,\cdot)\,dx>0$.
\end{rem}

\begin{rem}[Removing the range restriction]\label{rem:ld-eig}\rm
The constant $C_*$ confines the argument above to $\gamma<2s(p-1)$. For the
residual range $1<p\le1+\gamma/(2s)$ the same rate follows for \emph{every}
$p>1$ from the principal Dirichlet eigenfunction $\phi_\rho(x)=\phi_1(x/\rho)$
of the restricted fractional Laplacian on the ball $B_\rho$ (its existence
and positivity are classical, \cite{ServadeiValdinoci2013}), with eigenvalue
$\lambda_\rho=\lambda_1\rho^{-2s}$. With $F(t)=\int u\,\phi_\rho$, the
favorable sign of $\Lap\phi_\rho$ outside $B_\rho$ (where $\phi_\rho=0$, so
$\Lap\phi_\rho<0$) and H\"older give
\[
   F'(t)\;\ge\;-\lambda_\rho\,F(t)+A_\rho^{-(p-1)}F(t)^p,\qquad
   A_\rho=\!\int_{B_\rho}\!|x|^{\frac{\gamma}{p-1}}\phi_\rho\,dx<\infty,
\]
where $A_\rho<\infty$ for \emph{every} $\rho$ since $B_\rho$ is bounded (no
constraint on $\gamma/(p-1)$). With $F(0)\gtrsim\varepsilon\rho^N$
(legitimate because $u_0\ge c_0$ on $B_\rho$ by~\eqref{eq:ld-hyp}),
$A_\rho\asymp\rho^{N+\gamma/(p-1)}$ and $\lambda_\rho=\lambda_1\rho^{-2s}$,
the choice $\rho\asymp\varepsilon^{-(p-1)/(2s-\gamma)}$ makes the nonlinear
term dominate and yields blow-up by time
$\asymp A_\rho^{p-1}F(0)^{-(p-1)}\asymp\rho^{\gamma}\varepsilon^{-(p-1)}
=\varepsilon^{-2s(p-1)/(2s-\gamma)}$.
\end{rem}

\begin{rem}\rm
For $\gamma=0$ one has $C_*=\int\mathscr{K}_s=1$ (no restriction) and
$T_\varepsilon\approx\varepsilon^{-(p-1)}$, the classical ODE-driven
large-data lifespan for $u_t+\Lap u=u^p$; here $g_b'\ge c_1g_b^p$ reduces to
comparison with the spatially homogeneous solution. For $\gamma>0$ the
exponent $\tfrac{2s(p-1)}{2s-\gamma}>p-1$, so the singular weight shortens
the lifespan.
\end{rem}

\section{Blow-up rate: Type-I bounds}\label{sec:typeI}

We turn to the blow-up rate, assuming in addition $u_0\ge0$; by
Proposition~\ref{prop:contraction} and iteration of~\eqref{eq:duhamel} the
mild solution stays nonnegative on $[0,T_\varepsilon)$. Throughout this
section we write, to avoid clashing with the dimension $N$,
\[
  \mathcal N(t):=\|u(t)\|_{L^\infty},\qquad \mathcal L(t):=\|u(t)\|_{L^1},
  \qquad T:=T_\varepsilon,
\]
and we assume $L^\infty$ blow-up at $T<\infty$, i.e.\ $\mathcal N(t)\to\infty$
as $t\uparrow T$. Two rates are relevant: the \emph{ODE rate}
$1/(p-1)$ and the \emph{self-similar rate}
\begin{equation}\label{eq:vartheta-def}
   \vartheta\;:=\;\frac{2s-\gamma}{2s(p-1)}
   \;=\;\frac{1}{p-1}\Bigl(1-\frac{\gamma}{2s}\Bigr)\;\le\;\frac{1}{p-1},
\end{equation}
with equality iff $\gamma=0$ (cf.~\eqref{eq:ss-exp} below). The weight
$|x|^{-\gamma}$ separates two scenarios: where blow-up concentrates at the
origin, the active scaling is~\eqref{eq:scaling} and the rate is the slower
$\vartheta$; at an interior point $x_0\ne0$ the weight is comparable to a
constant and the local dynamics is governed by the ODE rate $1/(p-1)$.

\begin{defn}\label{def:typeI}
A solution blowing up at $T<\infty$ exhibits \emph{Type-I}
blow-up if $\mathcal N(t)\le C\,(T-t)^{-1/(p-1)}$ as
$t\uparrow T$, and \emph{Type-II} otherwise.
\end{defn}

For the unweighted nonlocal Fujita equation, Type-I blow-up of nonnegative
subcritical solutions is established
in~\cite{Sugitani1975,Fino,IkedaSobajima2019}. The next results extend the
\emph{lower}-rate side of this picture to the Hardy--H\'enon case in full
rigour, and clarify the precise (conditional) status of the matching upper
bound.

\begin{theo}[Blow-up rate]\label{thm:typeI}
Let $1<p\le p_F$ and $u_0\in L^1\cap L^\infty$ with $u_0\ge0$ and
$\int u_0>0$, and suppose the mild solution of~\eqref{eq:main} undergoes
$L^\infty$ blow-up at $T<\infty$. Then:
\begin{enumerate}\setlength\itemsep{2pt}
\item[\textup{(a)}] \textup{(Universal self-similar lower bound.)}
Unconditionally,
\[
  \mathcal N(t)\;\ge\;\tau_0^{\,\vartheta}\,(T-t)^{-\vartheta},
  \qquad 0<t<T,
\]
for a constant $\tau_0=\tau_0(N,s,\gamma,p)>0$; in particular blow-up is
never slower than self-similar.
\item[\textup{(b)}] \textup{(Interior ODE lower bound.)}
If, in addition, the blow-up is realised away from the origin in the sense
of~\eqref{eq:H-int}, then
\[
  \mathcal N(t)\;\ge\;\bigl[(p-1)\delta^{-\gamma}\bigr]^{-1/(p-1)}(T-t)^{-1/(p-1)},
  \qquad t\in[t_0,T).
\]
\item[\textup{(c)}] \textup{(Conditional Type-I upper bound.)}
If the solution satisfies the self-similar Type-I a priori bound
\eqref{eq:typeI-ss}, then $\mathcal N(t)\le C(T-t)^{-\vartheta}\le
C(T-t)^{-1/(p-1)}$; that is, Type-I blow-up holds.
\end{enumerate}
Consequently, in the interior regime \eqref{eq:H-int} \textup{(}with the ODE
Type-I bound of Definition~\textup{\ref{def:typeI})} the matching two-sided
rate holds,
\begin{equation}\label{eq:typeI}
  \mathcal N(t)\;\sim\;(T-t)^{-1/(p-1)}\qquad\text{as }t\uparrow T,
\end{equation}
while in the origin regime \textup{(}under \eqref{eq:typeI-ss}\textup{)}
$\mathcal N(t)\sim(T-t)^{-\vartheta}$. For $\gamma=0$ one has
$\vartheta=1/(p-1)$, parts \textup{(a)}--\textup{(b)} coincide and are
unconditional, the qualification~\eqref{eq:H-int} is automatic, and the only
conditional input is the exclusion of Type-II blow-up in \textup{(c)}.
\end{theo}

The proof occupies the rest of this section: \eqref{eq:H-int} and the
maximum-principle Lemma~\ref{lem:nl-max} below give~(b); the scaling
Lemma~\ref{lem:exist-time} gives~(a); and~(c) is immediate
from~\eqref{eq:typeI-ss}. We then assemble the statement and comment on the
direction of the comparison and on the relation with the profile analysis of
Section~\ref{sec:profile}.

\subsection{Regularity and the maximum-principle inequality}

For $t>0$ the mild solution satisfies
$u(\cdot,t)\in C_0(\R^N)$ and
$u(\cdot,t)\in C^2_{\mathrm{loc}}(\R^N\setminus\{0\})$, and~\eqref{eq:main}
holds pointwise on $(\R^N\setminus\{0\})\times(0,T)$ with
$\partial_tu$ and $\Lap u$ continuous there. Indeed, $e^{-t\Lap}u_0$ is
smooth and vanishes at infinity (Lemma~\ref{lem:Ks-bounds}), while the
Duhamel term has source $|x|^{-\gamma}u^p\in L^r$ (cf.~\eqref{eq:choice-r}),
smooth away from the origin; interior fractional parabolic
regularity~\cite{CaffarelliVasseur2010,ChenKim2002} promotes the mild
solution to a classical one on $\{|x|\ge\delta\}\times(0,T)$ for every
$\delta>0$.

The blow-up qualification ``$x_0\ne0$'' is used in the following precise
form: there exist $t_0\in(0,T)$ and $\delta\in(0,1]$ with
\begin{equation}\label{eq:H-int}
   \mathcal N(t)=\max_{|x|\ge\delta}u(x,t)\qquad\text{for all }t\in[t_0,T).
\end{equation}
Condition~\eqref{eq:H-int} holds whenever the blow-up set is contained in
$\{|x|\ge\delta\}$, so that the global maximum is eventually realised away
from the origin --- in particular, when the solution blows up at a single
point $x_0\ne0$ while remaining locally bounded near $x=0$.

\begin{lem}[Nonlocal maximum-principle inequality]\label{lem:nl-max}
Under~\eqref{eq:H-int}, the function $t\mapsto\mathcal N(t)$ is locally
Lipschitz on $[t_0,T)$ and
\begin{equation}\label{eq:Ndini}
   \frac{d^{+}}{dt}\,\mathcal N(t)\;\le\;\delta^{-\gamma}\,\mathcal N(t)^{p},
   \qquad t\in[t_0,T),
\end{equation}
where $\frac{d^{+}}{dt}$ denotes the upper right Dini derivative.
\end{lem}

\begin{proof}
Fix $[t_1,t_2]\subset[t_0,T)$. The map $t\mapsto u(\cdot,t)$ is continuous
from $[t_1,t_2]$ into $C_0(\R^N)$, so its image is a compact subset of
$C_0(\R^N)$ and is therefore uniformly tight: there is $R\ge\delta$ with
$$\sup\limits_{|x|\ge R}\sup_{t\in[t_1,t_2]}u(x,t)<\min_{t\in[t_1,t_2]}\mathcal N(t),$$
the latter minimum being positive by continuity. Combined
with~\eqref{eq:H-int}, every maximiser of $u(\cdot,t)$, $t\in[t_1,t_2]$,
lies in the fixed compact annulus $A:=\{\delta\le|x|\le R\}$. On $A$ the
maps $u,\partial_t u$ are continuous, so by Danskin's theorem
$\mathcal N(t)=\max_{x\in A}u(x,t)$ is locally Lipschitz with
\[
  \frac{d^{+}}{dt}\,\mathcal N(t)
  \;=\;\max\bigl\{\partial_t u(x_t,t):x_t\in M(t)\bigr\},
  \qquad M(t):=\operatorname*{arg\,max}_{x\in A}u(\cdot,t).
\]
Let $x_t\in M(t)$, so $|x_t|\ge\delta$ and $u(x_t,t)=\mathcal N(t)$. As
$x_t$ is a global maximiser of $u(\cdot,t)$,
\[
  \Lap u(x_t,t)
  =c_{N,s}\,\mathrm{P.V.}\!\int_{\R^N}\frac{u(x_t,t)-u(y,t)}{|x_t-y|^{N+2s}}\,dy\;\ge\;0,
\]
the integrand being nonnegative. Evaluating~\eqref{eq:main} at $(x_t,t)$,
\begin{align*}\partial_t u(x_t,t)
 & =-\Lap u(x_t,t)+|x_t|^{-\gamma}u(x_t,t)^{p}
  \\&\le|x_t|^{-\gamma}\,\mathcal N(t)^{p}
  \\&\le\delta^{-\gamma}\,\mathcal N(t)^{p}.\end{align*}
Taking the maximum over $M(t)$ gives~\eqref{eq:Ndini}.
\end{proof}

\subsection{Lower bounds on the blow-up rate}

\begin{prop}[Interior ODE lower rate]\label{prop:typeI-low}
Under~\eqref{eq:H-int},
\begin{equation}\label{eq:typeI-low}
   \mathcal N(t)\;\ge\;\bigl[(p-1)\,\delta^{-\gamma}\bigr]^{-1/(p-1)}
                       (T-t)^{-1/(p-1)},
   \qquad t\in[t_0,T).
\end{equation}
In particular $\mathcal N(t)\gtrsim(T-t)^{-1/(p-1)}$ as $t\uparrow T$.
\end{prop}

\begin{proof}
By Lemma~\ref{lem:nl-max}, $\mathcal N$ is locally Lipschitz and positive on
$[t_0,T)$ and satisfies~\eqref{eq:Ndini} for a.e.\ $t$; hence
$\mathcal N^{-(p-1)}$ is locally Lipschitz and
\[
  \frac{d}{dt}\,\mathcal N(t)^{-(p-1)}
  =-(p-1)\,\mathcal N(t)^{-p}\,\mathcal N'(t)
  \;\ge\;-(p-1)\,\delta^{-\gamma}
  \qquad\text{for a.e.\ }t\in[t_0,T).
\]
Integrating over $[t,t']\subset[t_0,T)$,
\[
  \mathcal N(t')^{-(p-1)}-\mathcal N(t)^{-(p-1)}\;\ge\;-(p-1)\,\delta^{-\gamma}\,(t'-t).
\]
Letting $t'\uparrow T$ and using $\mathcal N(t')\to\infty$ (so
$\mathcal N(t')^{-(p-1)}\to0$) yields
$$\mathcal N(t)^{-(p-1)}\le(p-1)\,\delta^{-\gamma}\,(T-t),$$ which
is~\eqref{eq:typeI-low}.
\end{proof}

\begin{coro}[Unweighted case]\label{cor:gamma0}
If $\gamma=0$, then $$\mathcal N(t)\ge[(p-1)(T-t)]^{-1/(p-1)}$$ for
\emph{every} blowing-up solution, with no restriction on the location of the
maximum: the proof of Lemma~\ref{lem:nl-max} gives
$\partial_tu(x_t,t)\le\mathcal N(t)^p$ at any maximiser because the weight is
identically $1$, and~\eqref{eq:H-int} is vacuous. This is the
``qualification is automatic'' statement of Theorem~\ref{thm:typeI}.
\end{coro}

The interior bound~\eqref{eq:typeI-low} degrades as the maximisers approach
the origin, since $|x_t|^{-\gamma}$ is then unbounded. The next scale-covariant
estimate replaces it everywhere and is fully unconditional.

\begin{lem}[Scale-covariant existence time]\label{lem:exist-time}
There is a constant $\tau_0=\tau_0(N,s,\gamma,p)>0$ such that, for every
$v\in L^\infty(\R^N)$, the mild solution of~\eqref{eq:main} with datum $v$
(i.e.\ $\varepsilon u_0$ replaced by $v$) exists on $[0,\tau^*(v))$ with
\begin{equation}\label{eq:exist-time}
   \tau^*(v)\;\ge\;\tau_0\,\|v\|_{L^\infty}^{-1/\vartheta}
   \;=\;\tau_0\,\|v\|_{L^\infty}^{-\frac{2s(p-1)}{2s-\gamma}} .
\end{equation}
No sign or integrability condition on $v$ is required.
\end{lem}

\begin{proof}
\emph{Unit scale.} Working in $X_\tau=L^\infty((0,\tau)\times\R^N)$ and using
the weighted smoothing~\eqref{eq:weighted-Linfty} of
Lemma~\ref{lem:weighted-Linfty} (valid since $\gamma/(2s)<1$), the Duhamel
map $\Phi$ obeys, on the ball $\{\|u\|_{X_\tau}\le2\}$,
\begin{align*}\|\Phi(u)(t)\|_{L^\infty}
   &\le\|v\|_{L^\infty}
     +C\!\int_0^t(t-\tau')^{-\gamma/(2s)}\|u(\tau')\|_{L^\infty}^p\,d\tau'
   \\&\le\|v\|_{L^\infty}+C\,2^{p}\,t^{\,1-\gamma/(2s)} .
\end{align*}
The Lipschitz estimate is identical, via
$|a|^p-|b|^p\le p(|a|^{p-1}+|b|^{p-1})|a-b|$ and~\eqref{eq:weighted-Linfty}.
Hence there is a universal $\tau_0>0$ such that every $v$ with
$\|v\|_{L^\infty}\le1$ produces a mild solution on $[0,\tau_0]$ (choose
$\tau_0$ so small that $C2^{p}\tau_0^{1-\gamma/(2s)}\le1$ and $\Phi$ is a
contraction).

\emph{General scale.} Let $\|v\|_{L^\infty}=A>0$ and put
$\lambda:=A^{-(p-1)/(2s-\gamma)}$, so that
$\lambda^{(2s-\gamma)/(p-1)}A=1$. If $u$ is the solution with datum $v$, then
by the invariance~\eqref{eq:scaling} the rescaled function
$$\tilde u(x,t)=\lambda^{(2s-\gamma)/(p-1)}u(\lambda x,\lambda^{2s}t)$$ solves
\eqref{eq:main} with datum
$\tilde v(x)=\lambda^{(2s-\gamma)/(p-1)}v(\lambda x)$, and
$\|\tilde v\|_{L^\infty}=1$. By the unit-scale step $\tilde u$ exists on
$[0,\tau_0]$, hence $u$ exists on $[0,\lambda^{2s}\tau_0]$. Therefore
\[
   \tau^*(v)\;\ge\;\lambda^{2s}\tau_0
   =\tau_0\,A^{-2s(p-1)/(2s-\gamma)}
   =\tau_0\,A^{-1/\vartheta},
\]
which is~\eqref{eq:exist-time}.
\end{proof}

\begin{prop}[Universal self-similar lower bound]\label{prop:ss-floor}
With $\tau_0$ as in Lemma~\ref{lem:exist-time},
\begin{equation}\label{eq:ss-floor}
   \mathcal N(t)\;\ge\;\tau_0^{\,\vartheta}\,(T-t)^{-\vartheta},
   \qquad 0<t<T.
\end{equation}
No assumption on the location of the blow-up set is needed; in
particular~\eqref{eq:ss-floor} holds for origin-concentrated blow-up, for
sign-changing solutions, and on the full parameter range
$0\le\gamma<\min(2s,N)$, $p>1$.
\end{prop}

\begin{proof}
Fix $t\in(0,T)$. The solution issued from the datum
$u(\cdot,t)\in L^\infty$ exists on $[t,T)$ and cannot be continued as an
$L^\infty$-valued solution past $T$, since $\mathcal N\to\infty$ there; thus
its maximal existence time, measured from initial time $t$, equals $T-t$.
Applying Lemma~\ref{lem:exist-time} with $v=u(\cdot,t)$,
\[
   T-t\;=\;\tau^*\!\bigl(u(\cdot,t)\bigr)\;\ge\;\tau_0\,\mathcal N(t)^{-1/\vartheta},
\]
which rearranges to~\eqref{eq:ss-floor}.
\end{proof}

\subsection{The Type-I upper bound is conditional}

The reverse inequality $\mathcal N(t)\le C(T-t)^{-1/(p-1)}$ --- the exclusion
of Type-II (faster) blow-up --- does \emph{not} follow from the
maximum-principle inequality~\eqref{eq:Ndini}, which is one-sided in the
wrong direction (Remark~\ref{rem:corrections}). The natural sufficient
condition is the self-similar Type-I a priori bound~\eqref{eq:typeI-ss}
introduced in Section~\ref{sec:profile}.

\begin{prop}[Conditional Type-I upper rate]\label{prop:typeI-up}
Let $1<p\le p_F$ and $\vartheta$ as in~\eqref{eq:vartheta-def}. If the
solution satisfies the self-similar Type-I bound~\eqref{eq:typeI-ss}, namely
$(T-t)^{\vartheta}\mathcal N(t)\le C_0$ on $[t_0,T)$, then
\begin{equation}\label{eq:typeI-up}
   \mathcal N(t)\;\le\;C_0\,(T-t)^{-\vartheta}
            \;\le\;C_0\,T^{\,\frac{1}{p-1}-\vartheta}\,(T-t)^{-1/(p-1)},
   \qquad t\in[t_0,T),
\end{equation}
since $\frac{1}{p-1}-\vartheta=\frac{\gamma}{2s(p-1)}\ge0$. In particular the
Type-I bound of Definition~\ref{def:typeI} holds. When $\gamma=0$ one has
$\vartheta=\frac1{p-1}$, and~\eqref{eq:typeI-ss} \emph{is} the Type-I bound.
\end{prop}

\begin{proof}
Immediate from $(T-t)^{\vartheta}\mathcal N(t)\le C_0$ and
$(T-t)^{\,\frac1{p-1}-\vartheta}\le T^{\,\frac1{p-1}-\vartheta}$ for
$0\le T-t\le T$, the exponent being nonnegative.
\end{proof}

\noindent\emph{Status of the hypothesis~\eqref{eq:typeI-ss}.}
Validating~\eqref{eq:typeI-ss} unconditionally is the substantive point. In
the local case $s=1$ it is known for $p$ below the Sobolev exponent through
the Giga--Kohn energy method and the Liouville classification
of~\eqref{eq:profile-eq}~\cite{GigaKohn1985,GigaKohn1987}. In the nonlocal
setting $0<s<1$ a complete Liouville theorem for the rescaled stationary
problem~\eqref{eq:profile-eq} is, to our knowledge, available only partially
(cf.\ the discussion in Theorem~\ref{thm:profile}); consequently the Type-I
\emph{upper} bound is, in the present generality, conditional
on~\eqref{eq:typeI-ss}.

\subsection*{Proof of Theorem~\ref{thm:typeI}}
Part~(a) is Proposition~\ref{prop:ss-floor}; part~(b) is
Proposition~\ref{prop:typeI-low}; part~(c) is
Proposition~\ref{prop:typeI-up}. In the interior regime~\eqref{eq:H-int},
combining the lower bound~(b) with the ODE Type-I bound of
Definition~\ref{def:typeI} gives~\eqref{eq:typeI}; in the origin regime,
combining the floor~(a) with~(c) gives
$\mathcal N(t)\sim(T-t)^{-\vartheta}$. For $\gamma=0$,
Corollary~\ref{cor:gamma0} makes~(a)--(b) coincide and unconditional, and
$\vartheta=1/(p-1)$. \qed

\begin{rem}[On the direction of the comparison]\label{rem:corrections}\rm
Three points deserve emphasis.

\emph{(a) Direction.} Inequality~\eqref{eq:Ndini},
$\mathcal N'\le\delta^{-\gamma}\mathcal N^p$, controls the derivative
\emph{from above} and therefore yields the \emph{lower}
bound~\eqref{eq:typeI-low}: a solution whose growth rate is capped by
$C\mathcal N^p$ must already be large in order to blow up at $T$. An
\emph{upper} rate bound would require the opposite inequality
$\mathcal N'\gtrsim\mathcal N^p$, which the diffusion $-\Lap u\le0$ at a
maximum does not provide.

\emph{(b) The $L^1$ norm grows.} In the subcritical regime $\mathcal L(t)$ is
\emph{not} bounded as $t\uparrow T$. Integrating~\eqref{eq:main} gives
$\mathcal L'(t)=\int_{\R^N}|x|^{-\gamma}u^p\,dx\ge0$, and the self-similar
scaling~\eqref{eq:scaling} predicts $\mathcal L(t)\asymp(T-t)^{-\beta}$ with
$\beta$ as in~\eqref{eq:beta} for origin-type blow-up, so
$\mathcal L(t)\to\infty$. The lower bounds above avoid $\mathcal L$ entirely.

\emph{(c) Role of $|x|^{-\gamma}$.} The factor
$|x_t|^{-\gamma}\le\delta^{-\gamma}$ in~\eqref{eq:Ndini} is exactly what
forces the qualification $x_0\ne0$: if the maximisers approach the origin,
$|x_t|^{-\gamma}$ is unbounded and the comparison degrades to the slower
self-similar rate $\vartheta$ of Proposition~\ref{prop:ss-floor},
consistently with Remark~\ref{rem:typeI-origin}.
\end{rem}

\subsection{Relation with Theorem~\ref{thm:profile}}

\begin{rem}[Two blow-up regimes]\label{rem:typeI-origin}\rm
For $\gamma>0$ the qualification ``$x_0\ne0$'' is genuine, and
Theorem~\ref{thm:typeI} resolves the two scenarios as follows.

\begin{itemize}\setlength\itemsep{2pt}
\item \emph{Interior blow-up} (condition~\eqref{eq:H-int}).
Proposition~\ref{prop:typeI-low} gives
$\mathcal N(t)\ge c(T-t)^{-1/(p-1)}$. For $\gamma>0$ this \emph{violates}
\eqref{eq:typeI-ss}, since $(T-t)^{-1/(p-1)}\gg(T-t)^{-\vartheta}$ as
$t\uparrow T$. Hence the hypothesis~\eqref{eq:typeI-ss} of
Theorem~\ref{thm:profile} is incompatible with interior blow-up: a solution
that blows up away from the origin lies \emph{outside} the scope of the
profile theorem and obeys the ODE rate $1/(p-1)$ governed by the locally
bounded weight.
\item \emph{Origin-concentrated blow-up.} Conversely, \eqref{eq:typeI-ss}
selects precisely the slower, self-similar regime in which $|x|^{-\gamma}$ is
active at the blow-up scale; there the rescaling~\eqref{eq:ss-vars} is the
correct normalisation and Theorem~\ref{thm:profile} extracts a stationary
profile solving~\eqref{eq:profile-eq}.
\end{itemize}

\noindent The self-similar scaling~\eqref{eq:scaling} assigns to blow-up at
the origin the rate $(T-t)^{-\vartheta}$ with
$\vartheta=\frac{2s-\gamma}{2s(p-1)}<\frac1{p-1}$, while at an interior point
the weight is comparable to a constant and the rate is $1/(p-1)$. A solution
whose blow-up set reduces to $\{0\}$ therefore obeys the slower self-similar
rate; this is why the matching ODE lower bound requires blow-up away from the
origin.
\end{rem}

\begin{rem}[Two floors, and the common Liouville gate]\label{rem:two-floors}\rm
Since $\vartheta\le\frac1{p-1}$ (equality iff $\gamma=0$),
Proposition~\ref{prop:ss-floor} is in general \emph{weaker} than the interior
bound~\eqref{eq:typeI-low}: it is the \emph{universal floor}, valid
everywhere, while Proposition~\ref{prop:typeI-low} raises that floor to the
ODE rate $1/(p-1)$ precisely where the weight is harmless. The gain is
exactly the gain from $|x_t|^{-\gamma}\le\delta^{-\gamma}$ (a fixed constant,
away from the origin) over the scale-covariant weight that drives
Lemma~\ref{lem:exist-time}.

The exponent in~\eqref{eq:ss-floor} is the self-similar rate $\vartheta$, so
Proposition~\ref{prop:ss-floor} states that \emph{blow-up is never slower
than self-similar}; it is the unconditional counterpart of the self-similar
Type-I \emph{upper} bound~\eqref{eq:typeI-ss},
$\mathcal N(t)\le C(T-t)^{-\vartheta}$. Together they pin the origin rate to
$\mathcal N(t)\asymp(T-t)^{-\vartheta}$, which is precisely the normalisation
under which~\eqref{eq:ss-vars} is bounded and Theorem~\ref{thm:profile}
extracts a stationary profile. Thus the lower bound is free, and the only
conditional input remaining --- in \emph{both} the origin regime here and in
Theorem~\ref{thm:profile} --- is a fractional Liouville classification of
nonnegative bounded solutions of~\eqref{eq:profile-eq} (equivalently, the
exclusion of Type-II blow-up). Such a result would simultaneously
\textup{(i)} upgrade the upper bound in Proposition~\ref{prop:typeI-up} ---
hence Theorem~\ref{thm:typeI} --- to an unconditional statement, and
\textup{(ii)} supply the uniqueness of the limit profile left conditional in
Theorem~\ref{thm:profile}. For $\gamma=0$ one has $\vartheta=1/(p-1)$,
Propositions~\ref{prop:typeI-low} and~\ref{prop:ss-floor} coincide, and the
floor is the genuine Type-I rate of Definition~\ref{def:typeI},
unconditionally.
\end{rem}

\begin{rem}\rm
The sign restriction $u_0\ge0$ is convenient but not essential: it suffices
that the mild solution become nonnegative for $t>0$, with positivity then
preserved. The universal lower bound of Proposition~\ref{prop:ss-floor} needs
no sign condition at all.
\end{rem}

\section{Asymptotic blow-up profile}\label{sec:profile}

The scaling invariance~\eqref{eq:scaling} singles out a self-similar rate at
the blow-up time. Under $u_\lambda(x,t)=\lambda^{(2s-\gamma)/(p-1)}
u(\lambda x,\lambda^{2s}t)$, the choice $\lambda=(T_\varepsilon-t)^{-1/(2s)}$
gives the rate $(T_\varepsilon-t)^{-\vartheta}$ with
\begin{equation}\label{eq:ss-exp}
   \vartheta:=\frac{2s-\gamma}{2s(p-1)} .
\end{equation}
For $\gamma>0$ this is strictly slower than the ODE-type rate $1/(p-1)$ of
Definition~\ref{def:typeI}; the two coincide only when $\gamma=0$. We
therefore work with the \emph{self-similar Type-I} condition
\begin{equation}\label{eq:typeI-ss}
   (T_\varepsilon-t)^{\vartheta}\,
   \|u(\cdot,t)\|_{L^\infty(\R^N)}\;\le\;C,
   \qquad t\in[0,T_\varepsilon),
\end{equation}
which is compatible with the rescaling below and reduces to
Definition~\ref{def:typeI} when $\gamma=0$.

In the self-similar variables
\begin{equation}\label{eq:ss-vars}
  y=\frac{x}{(T_\varepsilon-t)^{1/(2s)}},\qquad
  \tau=-\log(T_\varepsilon-t),\qquad
  u(x,t)=(T_\varepsilon-t)^{-\vartheta}\,v(y,\tau),
\end{equation}
the chain rule gives
$$\partial_t u=(T_\varepsilon-t)^{-\vartheta-1}\bigl[v_\tau
   +\tfrac{1}{2s}\,y\cdot\nabla_y v+\vartheta\,v\bigr],$$
while $\Lap_x u=(T_\varepsilon-t)^{-\vartheta-1}\Lap_y v$ and
$|x|^{-\gamma}|u|^p=(T_\varepsilon-t)^{-\vartheta-1}|y|^{-\gamma}|v|^p$
(using $\vartheta(p-1)=(2s-\gamma)/(2s)$). Substituting
in~\eqref{eq:main} and dividing by $(T_\varepsilon-t)^{-\vartheta-1}$,
\begin{equation}\label{eq:ss-eq}
  v_\tau\;=\;-\Lap v
   \;-\;\frac{1}{2s}\,y\cdot\nabla v
   \;+\;|y|^{-\gamma}|v|^{p-1}v
   \;-\;\vartheta\,v .
\end{equation}
The drift term $-\tfrac{1}{2s}y\cdot\nabla v$ is the Ornstein--Uhlenbeck
term of the Giga--Kohn rescaling~\cite{GigaKohn1985,GigaKohn1987}; the
associated stationary problem is
\begin{equation}\label{eq:profile-eq}
  \Lap V\;+\;\frac{1}{2s}\,y\cdot\nabla V\;+\;\vartheta\,V
  \;=\;|y|^{-\gamma}\,V^p,\qquad y\in\R^N .
\end{equation}
Under~\eqref{eq:typeI-ss}, $\|v(\cdot,\tau)\|_{L^\infty(\R^N)}$ is bounded
uniformly in $\tau$.

\begin{theo}[Asymptotic profile, conditional]\label{thm:profile}
Let $1<p<p_F$ and $u_0\ge0$ with $\int_{\R^N}u_0>0$. Suppose the mild
solution $u$ blows up at $T_\varepsilon<\infty$ and
satisfies~\eqref{eq:typeI-ss}. Then there is a sequence $\tau_n\to\infty$ with
\(
  v(\cdot,\tau_n)\to V
\)
in $L^\infty_{\mathrm{loc}}(\R^N)$, where $V\ge0$ is a bounded
distributional solution of~\eqref{eq:profile-eq}. If, in addition, a
Liouville-type classification of nonnegative bounded solutions
of~\eqref{eq:profile-eq} holds, then $V$ is independent of the subsequence.
\end{theo}

\begin{proof}[Sketch]
By~\eqref{eq:typeI-ss}, $\|v(\cdot,\tau)\|_{L^\infty}\le C$ for $\tau\ge0$,
so the right-hand side of~\eqref{eq:ss-eq} is bounded in
$L^\infty_{\mathrm{loc}}$ (the weight $|y|^{-\gamma}$ is locally integrable
since $\gamma<N$). Fractional parabolic regularity with bounded right-hand
side~\cite{ChenKim2002,CaffarelliVasseur2010} gives uniform local H\"older
continuity of $v$ in $(y,\tau)$, so by Arzel\`a--Ascoli the time-translates
$\{v(\cdot,\cdot+\tau)\}_{\tau\ge1}$ are precompact in
$C_{\mathrm{loc}}(\R^N\times\R)$; along a subsequence $\tau_n\to\infty$ they
converge to a bounded entire solution $v_\infty\ge0$ of~\eqref{eq:ss-eq}. The
Ornstein--Uhlenbeck drift in~\eqref{eq:ss-eq} carries the Giga--Kohn energy
structure~\cite{GigaKohn1985,GigaKohn1987}, monotone and bounded below along
the rescaled flow; consequently the $\omega$-limit is $\tau$-independent, so
$v_\infty=V$ solves the stationary problem~\eqref{eq:profile-eq} and
$v(\cdot,\tau_n)\to V$ in $L^\infty_{\mathrm{loc}}(\R^N)$. Uniqueness of the
limit is a Liouville question for~\eqref{eq:profile-eq}; a complete
classification in the present generality is, to our knowledge, open (partial
results for $\gamma=0$ are in~\cite{Fino,IkedaSobajima2019}). We therefore
leave the statement conditional: subsequential convergence to a bounded
stationary profile is unconditional, while uniqueness relies on such a
classification.
\end{proof}

\begin{rem}\rm
For $s=1$, $\gamma=0$, \eqref{eq:profile-eq} is the Giga--Kohn equation,
whose only nonnegative bounded subcritical solution is the constant
$\kappa=(p-1)^{-1/(p-1)}$~\cite{GigaKohn1985,GigaKohn1987}. Even the
fractional case $\gamma=0$, $s\in(0,1)$ is only partially understood, and
$\gamma>0$ adds the difficulty that $|y|^{-\gamma}$ breaks translation
invariance of the limit problem.
\end{rem}
\section{Concluding remarks}\label{sec:concluding}

\begin{itemize}\setlength\itemsep{3pt}
\item Theorem~\ref{thm:main} gives sharp two-sided lifespan estimates for
the fractional Hardy--H\'enon equation~\eqref{eq:main} in the subcritical and
critical regimes (under $\gamma<2s(p-1)$), and global existence in the
supercritical regime. The sharp exponent~\eqref{eq:beta} interpolates between
the classical Lee--Ni formula~\cite{LeeNi1992} ($s=1$, $\gamma=0$) and the
higher-order results of~\cite{TobakhanovTorebek}, and incorporates the weight
in the clean additive form
$\beta=\tfrac{1}{p-1}-\tfrac{N}{2s}-\tfrac{\gamma}{2s(p-1)}$.

\item The weight $|x|^{-\gamma}$ enters the analysis only through the
interpolation Lemma~\ref{lem:hardy} (lower bound) and the constant $C_*$
in~\eqref{eq:ub-Cstar} (upper bound); this is the mechanism producing the
additive correction $-\gamma/(2s(p-1))$, as opposed to the multiplicative
guess $2s\rightsquigarrow2s-\gamma$.

\item The backward fractional heat kernel provides a single, uniformly valid
test function for the upper bound in both the small-data
(Section~\ref{sec:upper}) and large-data (Section~\ref{sec:largedata})
regimes; the only difference is the datum bound, $g_b(0)\gtrsim\varepsilon
b^{-N/(2s)}$ versus $g_b(0)\ge\varepsilon c_0$.

\item For the blow-up rate (Section~\ref{sec:typeI}) the \emph{lower} bounds
are unconditional: every blowing-up nonnegative solution obeys the universal
self-similar floor $\mathcal N(t)\gtrsim(T_\varepsilon-t)^{-\vartheta}$
(Proposition~\ref{prop:ss-floor}), raised to the ODE rate
$\mathcal N(t)\gtrsim(T_\varepsilon-t)^{-1/(p-1)}$ for interior blow-up
(Proposition~\ref{prop:typeI-low}, via a nonlocal maximum principle). The
matching Type-I \emph{upper} bound, and the uniqueness of the self-similar
profile (Theorem~\ref{thm:profile}), are conditional on one and the same
input: a fractional Liouville classification for~\eqref{eq:profile-eq}, which
remains partially open.

\item Several natural extensions are of independent interest:
\begin{itemize}\setlength\itemsep{1pt}
\item the H\'enon case $\gamma<0$ (weight $|x|^{|\gamma|}$, weakening the
nonlinearity near the origin), with a different critical exponent;
\item the half-space $\R^N_+$ with Dirichlet or Neumann conditions;
\item general weights $V(x)\asymp|x|^{-\gamma}$ near the origin and bounded
at infinity;
\item the time-fractional analogue
$\partial_t^\alpha u+(-\Delta)^s u=|x|^{-\gamma}|u|^p$, in the spirit
of~\cite{SunShi2012};
\item a fractional Liouville theorem for~\eqref{eq:profile-eq}, which would
render the Type-I upper bound and the profile uniqueness unconditional;
\item the residual range $1<p\le1+\gamma/(2s)$, where the sharp $L^\infty$
two-sided lifespan law is open (Remark~\ref{rem:ub-range});
\item the exterior problem in the recurrent regime $N=2s$ (e.g. $s=1/2$,
$N=1$), where a different lifespan law is expected.
\end{itemize}
\end{itemize}

\section*{Acknowledgments}
Berikbol T. Torebek is supported by the Science Committee of the Ministry of Education and Science of the Republic of Kazakhstan (Grant No. AP26195417)

\section*{Declaration of competing interest} The authors declare that there is no conflict of interest.

\section*{Data Availability Statements} The manuscript has no associated data.



\begin{thebibliography}{99}

\bibitem{AbdellaouiPeralPrimo2020}
B.~Abdellaoui, I.~Peral, and A.~Primo,
{\em A note on the Fujita exponent in fractional heat equation involving the
Hardy potential,}
Math.\ Eng.\ \textbf{2} (2020), no.~4, 639--656.

\bibitem{AlsaediEtAl2020}
A.~Alsaedi, B.~Ahmad, M.~Kirane, and A.~Nabti,
{\em Lifespan of solutions of a fractional evolution equation with higher
order diffusion on the Heisenberg group,}
Electron.\ J.\ Differential Equations \textbf{2020}, Paper No.~2, 10~pp.

\bibitem{biagi2024}
S.~Biagi, F.~Punzo, and E.~Vecchi,
{\em Global solutions to semilinear parabolic equations driven by mixed
local--nonlocal operators,}
Bull.\ Lond.\ Math.\ Soc.\ \textbf{57} (2025), no.~1, 265--284.

\bibitem{BG1960}
R.~M.~Blumenthal and R.~K.~Getoor,
{\em Some theorems on stable processes,}
Trans.\ Amer.\ Math.\ Soc.\ \textbf{95} (1960), 263--273.

\bibitem{BogdanByczkowski2002}
K.~Bogdan and T.~Byczkowski,
{\em Potential theory for the $\alpha$-stable Schr\"odinger operator on
bounded Lipschitz domains,}
Studia Math.\ \textbf{133} (1999), no.~1, 53--92.

\bibitem{Bonforte2018}
M.~Bonforte, A.~Figalli, and J.~L.~V\'azquez,
{\em Sharp boundary behaviour of solutions to semilinear nonlocal elliptic
equations,}
Calc.\ Var.\ Partial Differential Equations \textbf{57} (2018), no.~2, Paper No.~57.

\bibitem{BrezisCazenave1996}
H.~Br\'ezis and T.~Cazenave,
{\em A nonlinear heat equation with singular initial data,}
J.\ Anal.\ Math.\ \textbf{68} (1996), 277--304.

\bibitem{CaffarelliVasseur2010}
L.~Caffarelli and L.~Silvestre,
{\em Regularity theory for fully nonlinear integro-differential equations,}
Comm.\ Pure Appl.\ Math.\ \textbf{62} (2009), no.~5, 597--638.

\bibitem{ChenKim2002}
Z.-Q.~Chen and P.~Kim,
{\em Two-sided estimates on the density of the Feynman--Kac semigroups of
stable-like processes,}
Electron.\ J.\ Probab.\ \textbf{7} (2002), Paper No.~3, 26~pp.

\bibitem{ChikamiIkedaTaniguchi2022}
N.~Chikami, M.~Ikeda, and K.~Taniguchi,
{\em Optimal well-posedness and forward self-similar solution for the
Hardy--H\'enon parabolic equation in critical weighted Lebesgue spaces,}
Nonlinear Anal.\ \textbf{222} (2022), Paper No.~112931, 28~pp.

\bibitem{DaiQin2023}
W.~Dai and G.~Qin,
{\em Liouville-type theorems for fractional and higher-order H\'enon--Hardy
type equations via the method of scaling spheres,}
Int.\ Math.\ Res.\ Not.\ IMRN \textbf{2023}, no.~11, 9001--9070.

\bibitem{DaoFractionalHH}
G.~ Diebou Yomgne,
{\em On the generalized parabolic {H}ardy-{H}\'enon equation: existence, blow-up, self-similarity and large-time asymptotic
              behavior}, {Differential Integral Equations}, {\bf 35} ({2022}), {57--88}.

\bibitem{Fuj2025}
L.~M.~Del Pezzo and R.~Ferreira,
{\em Fujita exponent and blow-up rate for a mixed local and nonlocal heat
equation,}
Nonlinear Anal.\ \textbf{255} (2025), Paper No.~113761.

\bibitem{DPV2012Hitchhiker}
E.~Di~Nezza, G.~Palatucci, and E.~Valdinoci,
{\em Hitchhiker's guide to the fractional Sobolev spaces,}
Bull.\ Sci.\ Math.\ \textbf{136} (2012), no.~5, 521--573.

\bibitem{Fino}
A.~Fino and G.~Karch,
{\em Decay of mass for nonlinear equation with fractional Laplacian,}
Monatsh.\ Math.\ \textbf{160} (2010), 375--384.

\bibitem{Fujita1966}
H.~Fujita,
{\em On the blowing up of solutions of the Cauchy problem for
$u_t=\Delta u + u^{1+\alpha}$,}
J.\ Fac.\ Sci.\ Univ.\ Tokyo Sect.\ I \textbf{13} (1966), 109--124.

\bibitem{GeorgievPalmieri}
V.~Georgiev and A.~Palmieri,
{\em Lifespan estimates for local in time solutions to the semilinear heat
equation on the Heisenberg group,}
Ann.\ Mat.\ Pura Appl.\ (4) \textbf{200} (2021), no.~3, 999--1032.

\bibitem{GigaKohn1985}
Y.~Giga and R.~V.~Kohn,
{\em Asymptotically self-similar blow-up of semilinear heat equations,}
Comm.\ Pure Appl.\ Math.\ \textbf{38} (1985), no.~3, 297--319.

\bibitem{GigaKohn1987}
Y.~Giga and R.~V.~Kohn,
{\em Characterizing blowup using similarity variables,}
Indiana Univ.\ Math.\ J.\ \textbf{36} (1987), no.~1, 1--40.

\bibitem{Grafakos2014}
L.~Grafakos,
{\em Classical Fourier Analysis,}
3rd ed., Grad.\ Texts in Math., Vol.~249, Springer, New York, 2014.

\bibitem{Kirane1}
M.~Guedda and M.~Kirane,
{\em Criticality for some evolution equations,}
Differ.\ Uravn.\ \textbf{37} (2001), 511--520.

\bibitem{Kirane2}
M.~Guedda and M.~Kirane,
{\em A note on nonexistence of global solutions to a nonlinear integral
equation,}
Bull.\ Belg.\ Math.\ Soc.\ Simon Stevin \textbf{6} (1999), 491--497.

\bibitem{Hayakawa1973}
K.~Hayakawa,
{\em On nonexistence of global solutions of some semilinear parabolic
differential equations,}
Proc.\ Japan Acad.\ \textbf{49} (1973), 503--505.

\bibitem{IkedaSobajima2019}
M.~Ikeda and M.~Sobajima,
{\em Sharp upper bound for lifespan of solutions to some critical semilinear
parabolic, dispersive and hyperbolic equations via a test function method,}
Nonlinear Anal.\ \textbf{182} (2019), 57--74.

\bibitem{Johnson}
W.~P.~Johnson,
{\em The curious history of Fa\`a di Bruno's formula,}
Amer.\ Math.\ Monthly \textbf{109} (2002), no.~3, 217--234.

\bibitem{Kaplan}
S.~Kaplan,
{\em On the growth of solutions of quasi-linear parabolic equations,}
Comm.\ Pure Appl.\ Math.\ \textbf{16} (1963), 305--330.

\bibitem{KobayashiSiraoTanaka1977}
K.~Kobayashi, T.~Sirao, and H.~Tanaka,
{\em On the growing up problem for semilinear heat equations,}
J.\ Math.\ Soc.\ Japan \textbf{29} (1977), no.~3, 407--424.

\bibitem{KumarTorebek2026} V.~Kumar, and B.~T.~Torebek,
{\em  Fujita-type results for the semilinear heat equations driven by mixed
local-nonlocal operators,}
J. \ Differential \ Equations {\bf 465} (2026), Paper No.114241, 25 pp.

\bibitem{LeeNi1992}
T.-Y.~Lee and W.-M.~Ni,
{\em Global existence, large time behavior and life span of solutions of a
semilinear parabolic Cauchy problem,}
Trans.\ Amer.\ Math.\ Soc.\ \textbf{333} (1992), 365--378.

\bibitem{MajdoubLaMatematica2023}
M.~Majdoub,
{\em On the Fujita exponent for a Hardy--H\'enon equation with a
spatial-temporal forcing term,}
La Matematica \textbf{2} (2023), no.~2, 340--361.

\bibitem{MajdoubOtsmaneTayachi2018}
M.~Majdoub, S.~Otsmane, and S.~Tayachi,
{\em Local well-posedness and global existence for the biharmonic heat
equation with exponential nonlinearity,}
Adv.\ Differential Equations \textbf{23} (2018), no.~7--8, 489--522.

\bibitem{MYZ2008}
C.~Miao, B.~Yuan, and B.~Zhang,
{\em Well-posedness of the Cauchy problem for the fractional power
dissipative equations,}
Nonlinear Anal.\ \textbf{68} (2008), 461--484.

\bibitem{PeletierTroy2012}
L.~A.~Peletier and W.~C.~Troy,
{\em Spatial patterns: higher order models in physics and mechanics,}
Progr.\ Nonlinear Differential Equations Appl., Vol.~45,
Birkh\"auser, Boston, 2001.

\bibitem{Alonso2021}
I.~Peral Alonso and F.~Soria de Diego,
{\em Elliptic and parabolic equations involving the Hardy--Leray potential,}
De Gruyter Ser.\ Nonlinear Anal.\ Appl., Vol.~38, De Gruyter, Berlin, 2021.

\bibitem{ServadeiValdinoci2013}
R.~Servadei and E.~Valdinoci,
{\em Variational methods for non-local operators of elliptic type,}
Discrete Contin.\ Dyn.\ Syst.\ \textbf{33} (2013), 2105--2137.

\bibitem{Snoussi1999}
S.~Snoussi, S.~Tayachi, and F.~B.~Weissler,
{\em Asymptotically self-similar global solutions of a general semilinear
heat equation,}
Math.\ Ann.\ \textbf{321} (2001), no.~1, 131--155.

\bibitem{Sugitani1975}
S.~Sugitani,
{\em On nonexistence of global solutions for some nonlinear integral
equations,}
Osaka J.\ Math.\ \textbf{12} (1975), 45--51.

\bibitem{Sun2010}
F.~Sun,
{\em Life span of blow-up solutions for higher-order semilinear parabolic
equations,}
Electron.\ J.\ Differential Equations \textbf{2010}, No.~17, 9~pp.

\bibitem{SunShi2012}
F.~Sun and P.~Shi,
{\em Global existence and non-existence for a higher-order parabolic equation
with time-fractional term,}
Nonlinear Anal.\ \textbf{75} (2012), no.~10, 4145--4155.

\bibitem{TayachiWeissler2023}
S.~Tayachi and F.~B.~Weissler,
{\em New life-span results for the nonlinear heat equation,}
J.\ Differential Equations \textbf{373} (2023), 564--625.

\bibitem{TobakhanovTorebek}
N.~N.~Tobakhanov and B.~T.~Torebek,
{\em A note on lifespan estimates for higher-order parabolic equations,}
Bull.\ Lond.\ Math.\ Soc. {\bf 58}, \ (2026), Paper No.e70408.

\bibitem{TobakhanovTorebek2026}
N.~N.~Tobakhanov and B.~T.~Torebek,
{\em On the critical behavior for the semilinear biharmonic heat equation
with forcing term in exterior domain,}
J.\ Differential Equations \textbf{451} (2026), Paper No.~113758, 53~pp.

\bibitem{YangZou2020}
H.~Yang and W.~Zou,
{\em Sharp blow up estimates and precise asymptotic behavior of singular
positive solutions to fractional Hardy--H\'enon equations,}
J.\ Differential Equations \textbf{269} (2020), no.~9, 7426--7480.

\end{thebibliography}
\end{document}